\documentclass[11pt]{article}
\usepackage{amsthm}
\usepackage{amsmath}
\usepackage{amssymb}
\usepackage{framed, fullpage}

\newtheorem{theorem}               {Theorem}
\newtheorem{lemma}{Lemma}[section]

\newtheorem{claim}[lemma]{Claim}
\newtheorem{fact}[lemma]{Fact}
\newtheorem{definition} [lemma] {Definition}

\newcommand{\ve}{\varepsilon}

\renewenvironment{proof}[1][\proofname]{\medskip\noindent{\bfseries #1: }}{\hfill$\blacksquare$\medskip}

\title{A Wowzer-Type Lower Bound for the Strong Regularity Lemma}
\author{
Subrahmanyam Kalyanasundaram\thanks{Department of Computer Science and Engineering,
    IIT Hyderabad, India. Email: {\tt subruk@iith.ac.in}. This
    work was done while being a student in School of Computer Science,
    Georgia Institute of Technology, Atlanta, GA 30332.}
\and Asaf Shapira\thanks{School of Mathematics, Tel-Aviv University, Tel-Aviv, Israel 69978,
and Schools of Mathematics and Computer Science, Georgia Institute of Technology,
Atlanta, GA 30332. Email: {\tt asafico@tau.ac.il}. Supported in part by NSF Grant
DMS-0901355, ISF Grant 224/11 and a Marie-Curie CIG Grant 303320.}
}

\begin{document}

\maketitle

\begin{abstract}
The regularity lemma of Szemer\'edi asserts that one can partition every graph into a bounded number of quasi-random bipartite graphs.
In some applications however, one would like to have a strong control on how quasi-random
these bipartite graphs are. Alon, Fischer, Krivelevich and Szegedy obtained a powerful variant of the regularity lemma, which allows one to have an {\em arbitrary}
control on this measure of quasi-randomness. However, their proof only guaranteed to produce a partition where the number of parts
is given by the Wowzer function, which is the iterated version of the Tower function.
We show here that a bound of this type is unavoidable by constructing a graph $H$, with the property that even if one wants a very mild
control on the quasi-randomness of a regular partition, then
the number of parts in any such partition of $H$ must be given by
a Wowzer-type function.
\end{abstract}

\section{Introduction}\label{sec:intro}

The regularity lemma of Szemer\'edi \cite{Sz} is one of the most widely used tools in extremal combinatorics. The lemma was originally
devised as part of Szemer\'edi's proof of his (eponymous) theorem \cite{Sztheo} on arithmetic progressions in dense sets of integers.
Since then it has turned into a fundamental tool in extremal combinatorics, with
applications in diverse areas such as theoretical computer science, additive number theory, discrete geometry and of course graph theory. We refer the reader to \cite{KSim} and its references for more details on the rich history and applications of the regularity lemma.

Let us turn to formally state the regularity lemma. For a graph $G=(V,E)$ and two disjoint vertex sets $A$ and $B$,
we denote by $e_G(A,B)$ the number of edges of $G$ with one vertex in $A$ and one in $B$.
The \emph{density} $d_G(A,B)$ of the pair
$(A,B)$ in the graph $G$ is

\begin{equation}\label{eqdensity}
d_G(A,B)=e_G(A,B)/|A||B|\;,
\end{equation}
that is, $d_G(A,B)$ is the fraction of pairs $(x,y) \in A \times B$ such that $(x,y)$ is an edge of $G$.
For $\gamma> 0$, we say that the pair $(A,B)$ in a graph $G$ is \emph{$\gamma$-regular} if
for any choice of $A'\subseteq A$ of size at least $\gamma|A|$ and $B' \subseteq B$ of size at least $\gamma |B|$,
we have $|d_G(A',B')-d_G(A,B)| \leq \gamma$. Note that a large random bipartite graph is $\gamma$-regular for all $\gamma >0$.
Thus we can think of $\gamma$ as measuring the quasi-randomness of the bipartite graph connecting $A$ and $B$; the smaller $\gamma$ is
the more quasi-random the graph is. We will sometimes drop the subscript $G$ in the above notations when the graph $G$ we are referring to is clear
from context.

Let ${\cal Z}=\{Z_1,\ldots,Z_k\}$ be a partition of $V(G)$ into $k$ sets. Throughout the paper, we will only consider partitions into sets
$Z_i$ of equal size\footnote{In some papers partitions of this type are called {\em equipartitions}.}.
We will refer to each $Z \in {\cal Z}$ as a {\em cluster} of
the partition ${\cal Z}$. The {\em order} of a partition is the number of clusters it has ($k$ above).
We will sometimes use $|{\cal Z}|$ to denote the order of ${\cal Z}$.
We say that a partition ${\cal Z}=\{Z_1,\ldots,Z_k\}$ {\em refines} another partition ${\cal Z}'=\{Z'_1,\ldots,Z'_{k'}\}$
if each cluster of ${\cal Z}$ is contained in one of the clusters of ${\cal Z}'$.

A partition ${\cal Z}=\{Z_1,\ldots,Z_k\}$ of $V(G)$ is said to be $\gamma$-regular if all but $\gamma k^2$ of the pairs
$(Z_i,Z_j)$ are $\gamma$-regular. Szemer\'edi's regularity lemma can be formulated as follows:

\begin{theorem}{\bf (Szemer\'edi \cite{Sz})}\label{thm:sz}
For any $\gamma >0$ and $t$ there is an integer $K=K(t,\gamma)$ with the following property; given a graph $G$ and
a partition ${\cal A}$ of $V(G)$ of order $t$, one can find a $\gamma$-regular partition ${\cal B}$ of $V(G)$ which refines ${\cal A}$ and satisfies
$|{\cal B}| \leq K$.
\end{theorem}

Let $T(x)$ be the function satisfying $T(0)=1$ and $T(x)=2^{T(x-1)}$ for $x \geq 1$. So $T(x)$ is a tower of $2$'s of height $x$.
Szemer\'edi's proof of the regularity lemma \cite{Sz} showed that the function $K(t,\gamma)$ can be bounded from above\footnote{We note that in essentially any application of Theorem \ref{thm:sz}, one takes $t$ to be (at least) $1/\gamma$ so some papers simply consider the function $K'(\gamma)=K(1/\gamma,\gamma)$. The reason is that one wants to avoid ``degenerate'' regular partitions into a very small number of parts, where most of the graph's edges will belong to the sets $V_i$ where
one has no control on the edge distribution.} by $T(1/\gamma^5)$. For a long
time it was not clear if one could obtain better upper bounds for $K(t,\gamma)$. Besides being a natural problem, further motivation came from the fact that
some fundamental results, such as Roth's Theorem \cite{Roth,RSz}, could be proved using the regularity lemma. Hence improved upper bounds for $K(t,\gamma)$ might have resulted in improved bounds for several other fundamental problems. In a major breakthrough, Gowers \cite{Gowers} proved that the Tower-type dependence is indeed necessary. He showed that for any $\gamma >0$ there is a graph where any $\gamma$-regular partition must have size at least
$T(1/\gamma^{1/16})$.

Gowers' lower bound \cite{Gowers} can be stated as saying that if one wants a regular partition of order $k$, then
there are graphs in which
the best quasi-randomness measure one can hope to obtain is merely
$1/(\log^{*}(k))^{16}$.
Suppose however that for some $f: \mathbb N \mapsto (0,1)$, we would like to find a partition of a graph of order $k$ that will be ``close''
to being $f(k)$-regular. Alon, Fischer, Krivelevich and Szegedy \cite{afks} formulated the following notion of being close to $f(k)$-regular.

\begin{definition}\label{def:reg}{\bf ($(\epsilon,f)$-regular partition)} Let $f$ be a function $f: \mathbb N \mapsto (0,1)$.
An $(\epsilon,f)$-regular partition
of a graph $G$ is a pair of partitions ${\cal A}=\{V_i~:~ 1 \leq i \leq k\}$ and ${\cal B}=\{U_{i,i'}~:~ 1 \leq i \leq k, 1 \leq i' \leq \ell\}$ of $V(G)$, where
${\cal B}$ is a refinement of ${\cal A}$ (with $U_{i,i'} \subseteq V_i$), satisfying the following:
\begin{enumerate}
\item ${\cal B}$ is $f(k)$-regular.
\item Say that a pair $(V_i,V_{j})$ of clusters of ${\cal A}$ is {\em good} if all but
at most $\epsilon \ell^2$ of pairs $1 \leq i', j' \leq \ell$ satisfy $| d(U_{i,i'}, U_{j,j'}) - d(V_i, V_{j}) | < \epsilon$. Then, at least
$(1-\epsilon){k \choose 2}$ of the pairs $(V_i,V_{j})$ are good.
\end{enumerate}
\end{definition}

One useful way of thinking about the above notion is to ``forget'' for a moment about the partition ${\cal B}$ and just treat partition ${\cal A}$ as an $f(k)$-regular
partition. One then tries to extract some useful information from the assumption that ${\cal A}$ itself is $f(k)$-regular.
Finally, one uses the second property of Definition \ref{def:reg}, which says that the two partitions are {\em similar}, in order to show
that the information deduced from the {\em assumption} that ${\cal A}$ is $f(k)$-regular can actually be deduced from the {\em fact} that ${\cal B}$ is $f(k)$-regular.

One of the main results of \cite{afks} was that given a graph $G$ and {\em any} function $f$, one can construct an $(\epsilon,f)$-regular partition of $G$ of bounded size.
This version of the regularity lemma is sometimes referred to as the {\em strong regularity lemma}. As we have mentioned above,
in order to avoid degenerate partitions we will assume henceforth that an $(\epsilon,f)$-regular partition has order at least $1/\epsilon.$

\begin{theorem}\label{thm:afks}{\bf (Alon et al. \cite{afks})}
For every $\epsilon >0$ and $f:\mathbb N \mapsto (0,1)$, there is an integer $S =S(\epsilon,f)$ such
that any graph $G =(V,E)$ has an $(\epsilon,f)$-regular partition $({\cal A},{\cal B})$ where
$1/\epsilon \leq |{\cal A}|,|{\cal B}| \leq S$.
\end{theorem}

Let us describe two cases where one needs to have a better control of the measure of quasi-randomness of a regular partition. A first example is when
proving certain variants of the graph removal lemma \cite{RSz}. In such a scenario we are given a regular partition and would like to be able to say that since the partition behaves in
a quasi-random way, then we can find ``small'' subgraphs that we expect to find in a truly random graph. The only problem is that as the ``small'' structure we are trying
to find becomes larger, we need the measure of quasi-randomness to decrease with it. Some examples where Theorem \ref{thm:afks} was used to overcome such difficulties can be
found in \cite{afks,ASher,ASS,ARS,KNR,SR}. We note that in some of these papers, Theorem \ref{thm:afks} was used with functions $f$ that go to zero extremely fast, so the ability to
apply the theorem with arbitrary functions was crucial.

Another example when one wants a better control of the measure of quasi-randomness is when the graph we are trying to partition is very sparse. It is not hard to see
that for the notion of $\gamma$-regularity to make sense, the graph should have density at least $\gamma$. A well known case where one is faced
with increasingly sparse graphs is in the proofs of the hypergraph regularity lemma, that were obtained independently by Gowers \cite{Go2}, by R\"odl et al. \cite{FR,NRS,RS3} and
later also by Tao \cite{Tao}. In those proofs, one is partitioning not only the vertices of the hypergraph (as in Theorem \ref{thm:sz}) but also the pairs
of vertices into quasi-random bipartite graphs. However, in the process these bipartite graphs become sparser so one needs to control their quasi-randomness as a
function of their density.
Out of the aforementioned proofs of the hypergraph regularity lemma, Tao's
proof actually uses Theorem \ref{thm:afks} in order to address this issue.
See the survey of Gowers \cite{Go2} for an excellent account of this.

We finally note that the strong regularity lemma is also related to the notion of the limit of convergent graph sequences defined and studied in \cite{LSetc}.
Without defining these notions explicitly, we just mention that many of the results mentioned above which were proved using Theorem \ref{thm:afks}, were later
reproved using graph limits, see e.g. Lov\'asz and Szegedy \cite{LS2}. Furthermore, some of the important properties of the limit of a convergent graph sequence, such as its uniqueness \cite{LS}, also hold for $(\epsilon,f)$-regular partitions, see \cite{ASS}. Hence, one can view an $(\epsilon,f)$-regular partition as the discrete analogue of the (analytic) limit of a convergent graph sequence.

Let $W(x)$ be the function satisfying $W(0)=1$ and $W(x)=T(W(x-1))$ for $x \geq 1$. So the function $W$ is an iterated version of the Tower function $T(x)$.
The function $W$ is sometimes referred to as the Wowzer\footnote{This name was coined by Graham, Rothschild and Spencer \cite{GRS}.} function (for obvious reasons).
The proof of Theorem \ref{thm:afks} in \cite{afks} gave a $W$-type upper bound for the function $S(\epsilon,f)$ in Theorem \ref{thm:afks}. As we have mentioned above, in some applications of this
lemma one uses functions $f$ that go to zero extremely fast. But in some cases, as was the case in \cite{afks}, one uses moderate functions like $f(x)=1/x^2$.
However, even when the function $f$ is $f(x)=1/x$, the upper bound given in \cite{afks} for the function $S(\epsilon,f)$ is (roughly) $W(1/\epsilon)$.
Hence it is natural to ask if better bounds can be obtained for such versions of Theorem \ref{thm:afks}. Our main result here is that a $W$-type dependence is unavoidable even
in this case.

\begin{theorem}\label{thm:main}
Set $f(x)=1/x$. For every small enough $\epsilon \leq c_0$ there is a graph $H$ with the following property:
If $({\cal A},{\cal B})$ is an $(\epsilon,f)$-regular partition of $H$, and\footnote{As we have mentioned before, in order to rule out degenerate
partitions (such as taking a partition into 1 set) we assume that $|{\cal A}| \geq 1/\epsilon$. A similar assumption was used in \cite{afks}, where
they assume that $f(x) \leq \epsilon$. These two assumptions are basically equivalent (recall that $f(x)=1/x$), but the one we use makes the notation somewhat simpler.} $|{\cal A}| \geq 1/\epsilon$, then $|{\cal A}| \geq W(\sqrt{\log(1/\epsilon)}/100)$.
\end{theorem}

An interesting aspect of our proof is that it gives the same lower bound even if one considers a much weaker condition than the second condition
in Definition \ref{def:reg}. What we show is that the lower bound of Theorem \ref{thm:afks} holds even if one wants only $\epsilon^{1/10}{k^2}$ of the pairs $(V_i,V_j)$ to be good. Observe that Definition \ref{def:reg} asks\footnote{We note that the application of Theorem \ref{thm:afks} in \cite{afks} (as well as in most other papers) critically relied on the partition having $(1-\epsilon){k \choose 2}$ good pairs.} for $(1-\epsilon){k \choose 2}$ good pairs! In other words, the lower bound holds even if one
is interested in having a very weak similarity\footnote{Recall the discussion following Definition \ref{def:reg}.} between the partitions ${\cal A}$ and ${\cal B}$.

Another interesting aspect of the proof of Theorem \ref{thm:main} is that by resetting the parameters appropriately, one can get $W$-type lower bounds for $(\epsilon,f)$-regularity
for any function $f: \mathbb N \mapsto (0,1)$ going to zero faster that $1/\log^{*}(x)$. Observe that this is not a caveat of the proof; when $f(x)=1/\log^{*}(x)$, Theorem \ref{thm:sz} can be formulated
as saying that any graph has an $(\epsilon,f)$-regular partition of order $T(1/\epsilon^5)$. Hence, one cannot obtain a $W$-type lower bound for $f$ of this type.
So we see that even if one wants to have a very weak relation between the order of ${\cal A}$ and the regularity measure of ${\cal B}$ (say, $1/\log\log(k)$) one would still have to use a partition of size given by a $W$-type function\footnote{But in such cases the bound might become $W(\log\log(1/\epsilon))$ or some other $W$-type function.}.

The ideas we use here in order to prove Theorem \ref{thm:main} appear to be useful also for proving $W$-type lower bounds
for the hypergraph regularity lemma \cite{FR,Go1,Go2,NRS,RS3,Tao}. As we explained above, in this case also one is faced with the need to control a measure of quasi-randomness
approaching $0$, and this seems to be the main reason why the current bounds for this lemma are of $W$-type. This investigation is part of
a joint work of the second author with Dellamonica and R\"odl.

\paragraph{Organization:}
The rest of the paper is organized as follows. In the following section we describe the graph $H$ which we use in proving Theorem \ref{thm:main}.
In Section \ref{sec:main} we give an overview of the proof, state the two key lemmas that are needed to prove Theorem \ref{thm:main} and then derive Theorem \ref{thm:main} from them.
In Section \ref{sec:trap} we prove several preliminary lemmas that we would later use in the proofs of the two key Lemmas. In
Sections \ref{sec:lem2} and \ref{sec:lem1} we prove the key lemmas stated in Section \ref{sec:main}.

\section{A Hard Graph for the Strong Regularity Lemma}\label{sec:construct}

In this section we describe the graph $H$ which will have the properties asserted in Theorem \ref{thm:main}.
The description will be somewhat terse; the reader can find in Section \ref{sec:main} an overview of the proof of Theorem \ref{thm:main},
which includes some intuition/motivation for the way we define $H$.

\subsection{A weighted reformulation of Theorem \ref{thm:main}}\label{subsec:reform}

Suppose $P$ is a weighted complete graph, where each edge $(x,y)$ is assigned a weight $d_P(x,y) \in [0,1]$.
For two sets $A,B$ define the weighted density between $A,B$
\begin{equation}\label{eq:weighted}
d_P(A,B)=\sum_{x \in A, y \in B}d_P(x,y)/|A||B|\;.
\end{equation}
Note that if we think of a graph as a weighted complete graph with $0/1$ weights then the above definition coincides
with the definition of $d_G(A,B)$ given in (\ref{eqdensity}). Also note that when $A=\{x\}$, $B=\{y\}$ are just two vertices
then $d_P(A,B)$ is just the weight $d_P(x,y)$ assigned to $(x,y)$ as above.
The following simple claim follows immediately from a standard application of Chernoff's inequality.

\begin{claim}\label{clm:weight}
Let $\zeta > 0$.
Suppose $P$ is a weighted complete graph with weights in $[0,1]$, and $H$ is a random graph, where each edge $(x,y)$ is chosen independently to be included in $H$
with probability $d_P(x,y)$.
Then with probability at least
$1/2$ we have
$$
|d_H(A,B)-d_P(A,B)| \leq \zeta\;,
$$
for all sets $A,B$ of size at least $20\zeta^{-2}\log(n)$.
\end{claim}

It is clear that we can prove Theorem \ref{thm:main} by constructing an arbitrarily large graph, such that the number
of vertices $n$ will be much larger than all the constants involved. Hence, by the above claim, we see that in order
to prove Theorem \ref{thm:main} it is enough to construct a {\em weighted} graph $H$ satisfying
the condition of the theorem with respect to the notion of $d(A,B)$ defined in (\ref{eq:weighted}).
The reason is that by Claim \ref{clm:weight}, if we have a weighted graph $H$ satisfying Theorem \ref{thm:main},
then a random graph generated as in Claim \ref{clm:weight} will satisfy the assertion
of Theorem \ref{thm:main} with high probability. Therefore, from this point and {\em throughout the paper} we will
focus on the construction of a {\em weighted} graph $H$ satisfying the condition of Theorem \ref{thm:main}.
Hence, from now on, whenever we talk about $d(A,B)$ we will be referring to the weighted density between
$A,B$ as in (\ref{eq:weighted}).

\subsection{A preliminary construction}  \label{subsec:gowconstruct}

In this subsection we describe the first step in defining the graph $H$ of Theorem \ref{thm:main}.
This graph will be a variant of the graph used by Gowers in \cite{Gowers}. We start with the following definition.

\begin{definition}\label{def:balanced}{\bf(Balanced Partitions)}
Let $M$ be an integer and suppose we have a sequence $(A_i, B_i)_{i=1}^m$
of (not necessarily distinct) partitions of $[M]$. We call this sequence of partitions \emph{balanced}
if for any distinct $j, j' \in [M]$, the number of $1 \leq i \leq m$ for which $j$ and $j'$
lie in the same set of the partition $(A_i, B_i)$ is at most $3m/4$.
\end{definition}

The following claim appears in \cite{Gowers}. For completeness, we will reproduce a simple proof later on in the paper (see Section \ref{sec:trap}).

\begin{claim} \label{lem:balanced}
Let $\phi(m) =  2^{\lceil m/16\rceil}$. Then for every $m \geq 1$
there exists a sequence of $m$ balanced partitions of $\phi(m)$.
\end{claim}

Let $T^{\phi}(x)$ be the function satisfying $T^{\phi}(0)=1$ and $T^{\phi}(x)=T^{\phi}(x-1)\phi(T^{\phi}(x-1))$ for $x\geq 1$, where $\phi(x)=2^{\lceil x/16\rceil}$ is the function defined in Claim \ref{lem:balanced}.
It is not hard to see that $T^{\phi}$ is a Tower-type function, and that in fact $T^{\phi}(x) \geq T(\lfloor x/2 \rfloor)$ (see Section \ref{sec:trap}).

Let us define a sequence of integers as follows. We set
\begin{equation}\label{eqw0}
w(1)=\lfloor \log\log (1/\epsilon) \rfloor\;,
\end{equation}
and define inductively
\begin{equation}\label{eqw}
w(x+1)=\lfloor \log\log(T^{\phi}(w(x))) \rfloor\;.
\end{equation}
It is also not hard to see that $w(x)$ has a $W$-type dependence on $x$. Specifically we will later (see Section \ref{sec:trap}) observe that:

\begin{claim}\label{clm:W} For every integer $x \geq 1$, we have $w(x) \geq W(\lfloor x/2 \rfloor)$.
\end{claim}

We now turn to define a graph $G$, which we will later modify in order to get the actual graph $H$ that will satisfy the
assertion of Theorem \ref{thm:main}. In order to define $G$ we will first define a sequence of partitions of the vertex set
of $G$. For simplicity we will identify the $n$ vertices of $G$ with the integers $[n]$.
So let $n \in \mathbb N$ and set $s=w(\frac{1}{48}\sqrt{\log(1/\epsilon)})$, where $w(x)$ is the function
defined in (\ref{eqw}). We set $m_0 = 1$ and for $1 \leq r \leq s$, let $m_r = m_{r-1} \phi(m_{r-1})$.
For each $ 0 \leq r \leq s$,
let $X_1^{(r)}, X_2^{(r)}, \ldots, X_{m_r}^{(r)}$ be a partition of $[n]$ into $m_r$ intervals of integers of
equal size\footnote{We assume that $n$ is such that it can be divided into equal sized parts of size $m_r$ for all $0 \leq r \leq s$.}. We will
later refer to this partition as {\em canonical partition} $\mathcal P_r$.
Thus at level $r$, we have a canonical partition $\mathcal P_r$ consisting of $m_r$ clusters.
So ${\cal P}_0$ is just the entire vertex set of $G$.
Note that using the notation we introduced above we have
\begin{equation}\label{eq:parorder}
|{\cal P}_r|=m_r=T^{\phi}(r)\;.
\end{equation}
A crucial observation that will be used repeatedly in the paper is
that for every $r < r'$, partition ${\cal P}_{r'}$ refines partition ${\cal P}_r$.

We finally arrive at the actual definition of $G$. We will start with the graph $G$ where each pair of vertices $(x,y)$ has weight $0$.
We will then go over the partitions ${\cal P}_1,{\cal P}_2,\ldots,{\cal P}_{s}$ one after the other, and in each case increase
the weight between some of the pairs $(x,y)$.

Consider some $r \geq 1$ and focus on ${\cal P}_r$ and ${\cal P}_{r-1}$. Let us simplify the notation a bit and set $m=m_{r-1}$, $M=\phi(m)$ and $m_r = Mm$.
So $m$ is the number of clusters of ${\cal P}_{r-1}$, $M$ is the number of clusters of ${\cal P}_{r}$ inside each cluster of
${\cal P}_{r-1}$, and $mM$ is the number of clusters of ${\cal P}_{r}$.
Let us use $X_1,\ldots,X_m$ to denote the $m$ clusters of ${\cal P}_{r-1}$. Also, for each $1 \leq i \leq m$ we use
$X_{i,1},\ldots,X_{i,M}$ to denote the $M$ clusters of ${\cal P}_{r}$ inside $X_i$.
Now, for each $1 \leq i \leq m$, let $(A'_{i,j}, B'_{i,j})^m_{j=1}$ be a sequence of balanced partitions of $[M]$.
Such a collection exists since $M=\phi(m)$ so Claim \ref{lem:balanced} can be used here.
One can think of each of these partitions as partitioning the clusters of ${\cal P}_{r}$ within cluster $X_i$.
Let $A_{i,j} = \cup_{t \in A'_{i,j}} X_{i,t}$ and $B_{i,j} = \cup_{t \in B'_{i,j}} X_{i,t} = X_i
\backslash A_{i,j}$. We now update the weights of $G$ as follows: If $(x,y) \in X_i \times X_j$, then we increase $d_G(x,y)$ by $4^{-r}/4^{\sqrt{\log(1/\epsilon)}}$ if and only if
$(x,y) \in A_{i,j} \times A_{j,i}$ or $(x,y) \in B_{i,j} \times B_{j,i}$.
We will later refer several times to the following observation.

\begin{fact}\label{ob1} For any $x,y \in V(G)$ we have $d_G(x,y) \leq 4^{-\sqrt{\log(1/\epsilon)}}$.
\end{fact}

\subsection{Adding {\em Traps} to $G$} \label{subsec:trap}

We will now need to modify the graph $G$ defined above in order to obtain the graph $H$ from Theorem \ref{thm:main}.
To this end we will need to define certain quasi-random graphs. Let $b' < b$ and consider two of the canonical partitions ${\cal P}_{b'}$ and ${\cal P}_b$ defined in the previous subsection. Suppose ${\cal P}_{b}$ has order $m_b$ and let $V$ be a set of $m_b$ vertices, where we identify
vertex $i \in V$ with cluster $X_i \in {\cal P}_b$. Note that with this interpretation in mind, one can think of a cluster $U \in {\cal P}_{b'}$ as a subset
of vertices $U' \subseteq V$, where vertex $j$ belongs to $U'$ if and only if cluster $X_j \in {\cal P}_b$ is a subset of $U$.
It follows that for every $b' < b$, partition ${\cal P}_{b'}$ defines a natural partition of $V$ into $m_{b'}$ subsets $U^{b'}_1,\ldots,U^{b'}_{m_{b'}}$ corresponding to its $m_{b'}$ clusters.

We now arrive at a critical definition. We will use $e(R,R')$ to denote the number of edges in a graph with one vertex in $R$ and another in $R'$, where edges in $R \cap R'$ are counted twice\footnote{Note that this
definition is compatible with the definition of $e(A,B)$ we used earlier, where we assumed that the sets $A,B$ are disjoint.}.

\begin{definition}\label{def:trap}{\bf (Trap)} Let ${\cal P}_b$, $m_b$, $V$  and the partitions $U^{b'}_1,\ldots,U^{b'}_{m_{b'}}$ be as above. Let ${\cal O}=(V,E)$ be an $m_b$-vertex graph on $V$. Then ${\cal O}$ is said to be a {\em trap} if it satisfies the following two conditions:
\begin{itemize}
\item For every pair of sets $R,R' \subseteq V({\cal O})$ of size $\lceil \sqrt{m_b}/4 \rceil$ we have
$$
\left|e(R,R')-\frac12|R||R'|\right| \leq \frac14|R||R'|\;.
$$
\item For every $b' < b$, for every $1 \leq i, j \leq m_{b'}$, every choice of $200 \leq k \leq \log (m_b)$, every choice of $R \subseteq U^{b'}_i$ of size $k^6$ and every choice of $R' \subseteq U^{b'}_j$ of size $\lceil |U^{b'}_j|/k \rceil$, we have
$$
\left|e(R,R')-\frac12|R||R'|\right| \leq \frac{1}{k^2}|R||R'|\;.
$$
\end{itemize}
\end{definition}

\bigskip

We will later prove the following (see Section \ref{sec:trap}).

\begin{claim}\label{trapsexist} There is a constant $C$, such that for every $m > C$, there exists a trap on $m$ vertices.
\end{claim}

We are now ready to describe the modifications needed to turn $G$ into the graph $H$.
We do the following for every integer $1 \leq g \leq \frac{1}{48}\sqrt{\log(1/\epsilon)}$; let $b=w(g)$ be the integer defined
in (\ref{eqw}), let $m_b$ be the order of ${\cal P}_b$ and let ${\cal O}_b=(V,E)$ be\footnote{Note
that since we only ask Theorem \ref{thm:main} to hold for small enough $\epsilon$, we can assume that $\epsilon$ is small enough so that already $m_{w(1)}=T^{\phi}(w(1))$ would be larger than $C$, thus allowing us to pick a trap via Claim \ref{trapsexist} (where $w(1)$ is defined in (\ref{eqw0})).} a trap on a vertex set $V$ of size $m_b$. Recall that we identify vertex $i \in V$ with cluster $X_i \in {\cal P}_b$.
We now modify $G$ as follows; for every pair of clusters $(X_i,X_j)$, if $(i,j) \in E({\cal O}_b)$ we increase by $4^{-g}$ the weight of every pair of vertices $(x,y) \in X_i \times X_j$.
If $(i,j) \not \in E({\cal O}_b)$ we do not increase the weight of $(x,y)$. Let us state the following fact to which we will later refer.

\begin{fact}\label{trapdensity}
The smallest weight used when placing a trap in $H$ is $4^{-\frac{1}{48}\sqrt{\log(1/\epsilon)}}$.
\end{fact}

Later on in the paper we will say that we have {\em placed} a trap on partition ${\cal P}_b$ if $b$ is one of the integers $w(1),\ldots,w(\frac{1}{48}\sqrt{\log(1/\epsilon)})$.
If a trap was placed on ${\cal P}_b$ and $(i,j)$ is an edge of the graph ${\cal O}_b$ that was used in the previous paragraph, then we will say that the pair $(X_i,X_j)$ {\em belongs} to the trap placed on ${\cal P}_b$.
Also, if $b=w(g)$, then we will refer to the trap placed on ${\cal P}_{b}$ as the $g^{th}$ trap placed in $H$.
Finally, if $(x,y) \in X_i \times X_j$ and $(X_i,X_j)$ belongs to the trap placed on ${\cal P}_{w(g)}$ then we will say that $(x,y)$ {\em received} an extra weight of $4^{-g}$ from the $g^{th}$ trap placed in $H$.

Using the above jargon, we can thus say that in order to obtain the graph $H$ from the graph $G$ we do
the following for every $1 \leq g \leq \frac{1}{48}\sqrt{\log(1/\epsilon)}$; setting $b=w(g)$, we place the $g^{th}$ trap on partition ${\cal P}_{b}$, by increasing the weight of $(x,y)$ by $4^{-g}$ if and only if
$(x,y) \in X_i \times X_j$ and $(X_i,X_j)$ belongs to the trap.

Let us draw some distinction between the way we assigned weights to edges in $G$ and the way we have done so when modifying $G$ to obtain $H$.
When defining $G$ we looked at {\em each} of the partitions ${\cal P}_r$, and for {\em every} $X_i,X_j \in {\cal P}_{r-1}$ added weight $4^{-r}/4^{\sqrt{\log(1/\epsilon)}}$ only
to {\em some} of the pairs $(x,y) \in X_i \times X_j$. More specifically, we considered the partitions of $X_i=A_{i,j} \cup B_{i,j}$ and $X_j=A_{j,i} \cup B_{j,i}$
and only added the weight $4^{-r}/4^{\sqrt{\log(1/\epsilon)}}$ when either $(x,y) \in A_{i,j} \times A_{j,i}$ or $(x,y) \in B_{i,j} \times B_{j,i}$. When adding the traps, we have only added weights
to {\em some} of the partitions ${\cal P}_b$, that is, those for which $b=w(g)$ for some $1 \leq g \leq \frac{1}{48}\sqrt{\log(1/\epsilon)}$. Moreover, when placing a trap on ${\cal P}_b$
we added weight $4^{-g}$ only to pairs $(x,y)$ connecting {\em some} of the pairs $(X_i,X_j)$ (those that belong to the trap). Finally, for each such pair $(X_i,X_j)$ we either added more weights to all the pairs $(x,y) \in X_i \times X_j$ or to none of them.

Another important distinction is the following; suppose $b=w(g)$. Then in $G$, the weight that was added to ${\cal P}_b$ was $4^{-b}/4^{\sqrt{\log(1/\epsilon)}}$ while the weight we added when placing a trap on ${\cal P}_b$ is $4^{-g}$. Since $w$ is a $W$-type function we see that the weights assigned in $G$ to a specific partition ${\cal P}_b$ are extremely small compared to those assigned to ${\cal P}_b$ when placing a trap on it (assuming a trap was placed on ${\cal P}_b$).

We also observe that for every pair of vertices $(x,y)$ of $H$, the total weight it can receive from all the traps we placed is bounded
by $1/4+1/16+\ldots < 1/3$. We also recall Fact \ref{ob1} stating that the total weight assigned to a pair $(x,y)$ in $G$ is bounded by $1/4^{\sqrt{\log(1/\epsilon)}}$.
This means that $d_H(x,y) \leq 1$, as needed for the application of Claim \ref{clm:weight}.


\section{Proof Overview, Key Lemmas and Proof of Theorem \ref{thm:main}}\label{sec:main}

Our goal in this section is fourfold; give an overview of the proof of Theorem \ref{thm:main}, describe the
main intuition behind the construction of $H$, state the two key lemmas that will be used
to prove Theorem \ref{thm:main} and finally derive Theorem \ref{thm:main} from these two lemmas.

Perhaps the best way to approach our construction of $H$ is to first consider the proof of Theorem \ref{thm:afks} in \cite{afks}.
For simplicity, let us consider the case $f(x)=1/x$; we start by taking ${\cal A}_1$ to be an arbitrary partition of $G$ of order $1/\epsilon$, and then apply Theorem
\ref{thm:sz} in order to find a $1/|{\cal A}_1|$-regular partition, ${\cal B}_1$, of $G$ which refines ${\cal A}_1$. Note that by definition, ${\cal A}_1$ and ${\cal B}_1$ satisfy
the first condition of Definition \ref{def:reg}, so if they also satisfy the second, then we are done. If they do not, then we set ${\cal A}_2$ to be ${\cal B}_1$ and use Theorem \ref{thm:sz} to find
a $1/|{\cal A}_2|$-regular partition, ${\cal B}_2$, of $G$ which refines ${\cal A}_2$. Note that ${\cal A}_2$ and ${\cal B}_2$ satisfy
the first property, so if they satisfy the second we are done. The process thus goes on till we end up with a pair of partitions ${\cal A}_i$, ${\cal B}_i$ which satisfy the second condition.
The main argument in \cite{afks} shows that this process must stop after (about) $1/\epsilon$ steps with a pair ${\cal A}_i$, ${\cal B}_i$ which satisfies the second condition, and also
(by definition) the first condition. Since the above proof applies Theorem \ref{thm:sz} repeatedly, where each time we take $1/\gamma$ to be the order of the previous partition, the bound
we obtain is of $W$-type.

Of course, if we want to have any chance of proving Theorem \ref{thm:main}, we need to come up with a graph for which the {\em proof} of Theorem \ref{thm:afks} will
produce a partition of $W$-size. Given the overview of this proof described above, the graph $H$ needs to have two properties: (1) For every $\gamma >0$, any $\gamma$-regular
partition of $H$ has size given by a Tower-type function; (2) one needs to iteratively apply Theorem \ref{thm:sz} a super-constant\footnote{To be precise,
in order to get a $W$-type lower bound the number of iterations needs to be larger than $W^{-1}(1/\epsilon)$.} number of times in order
to get two partitions ${\cal A}$ and ${\cal B}$ satisfying the second condition of Definition \ref{def:reg}. The first property will guarantee that each time we apply Theorem
\ref{thm:sz} we get a Tower-type increase in the size of ${\cal A}_i$ while the second condition will guarantee that we will have to repeat this sufficiently many times.

Let us describe how to get a graph satisfying property (1) mentioned above. Recall that Gowers showed \cite{Gowers} that for
every $\gamma$ there exists a graph with the property that any $\gamma$-regular partition has a size $T(1/\gamma^{1/16})$. It is not hard to see that by a minor ``tweak'' of his
construction\footnote{In fact, we will be tweaking the construction of Gowers \cite{Gowers} which gives a slightly weaker lower bound of $T(\log(1/\gamma))$, and is much simpler to analyze. Since we are trying to prove $W$-type lower bounds it makes little difference if we are iterating the function $T(x)$ or $T(\log(x))$.} one can get a {\em single} graph that works for all $\gamma$ bounded away from 0. This is basically\footnote{If we were only interested in getting a graph
that for all $\gamma >0$ had only $\gamma$-regular partitions of Tower-size, then we could have used the weights $4^{-r}$ instead of $4^{-r}/4^{\sqrt{\log(1/\epsilon)}}$ like we do.} the graph $G$ we defined in Subsection \ref{subsec:gowconstruct}. For completeness let us describe the intuition behind Gowers' construction. Let us explain why the partitions ${\cal P}_{r}$ used in the construction of $G$ cannot be used as $\gamma$-regular partitions of $G$. Recall that at each iteration, we take every pair of sets $X_i,X_j \in {\cal P}_{r-1}$ split them as $X_i=A_{i,j} \cup B_{i,j}$ and $X_j = A_{j,i} \cup B_{j,i}$ and increase the weight between $A_{i,j},A_{j,i}$ and $B_{i,j},B_{j,i}$.
So, in some sense, each partition ${\cal P}_r$ is used in order to rule out the possibility of using the previous partition ${\cal P}_{r-1}$ as a $\gamma$-regular partition. We note that when one comes about to actually prove
that no other (small) partition can be $\gamma$-regular one relies critically on the fact that the weights assigned to the partitions ${\cal P}_r$ in $G$ decrease exponentially (as a function of $r$). This makes sure that any irregularity found in level $r$ cannot be canceled by weights
assigned to levels $r' > r$.

Let us describe how to get a graph satisfying property (2) mentioned above. Recall that $G$ was defined over a sequence of partitions ${\cal P}_r$.
Suppose we want to make sure that two specific partitions in this sequence ${\cal P}_r$ and ${\cal P}_{r'}$, with ${\cal P}_{r'}$ refining ${\cal P}_r$, will not satisfy the second property of Definition \ref{def:reg}.
Then we can do the following; we take a random graph ${\cal O}$ whose vertices are the clusters of ${\cal P}_{r'}$, and for every edge $(i',j') \in E({\cal O})$ increase the weight of all pairs $(x,y) \in U_{i'} \times U_{j'}$, where
$U_{i'},U_{j'} \in {\cal P}_{r'}$. This is just the trap we used in Subsection \ref{subsec:trap}.
Since we use a random graph, we expect {\em all} pairs of clusters $(X_i,X_j)$ of ${\cal P}_{r}$ to not be good (in the sense of Definition \ref{def:reg}) since close to half of the clusters $(U_{i'},U_{j'})$ with $U_{i'} \subseteq X_i,U_{j'} \subseteq X_j$, will get an extra
weight while the other half will not. Now it is not hard to see that for this to work we do not actually have to put the trap on ${\cal P}_{r'}$; it is enough to do that on some partition ${\cal P}_b$ with $r \leq b \leq r'$. Since we will make sure that a $\gamma$-regular partition must be huge, in order to satisfy the first condition of Definition \ref{def:reg} one would
have to pick two partitions ${\cal P}_{r'}$, ${\cal P}_r$ with $r'$ being much larger than $r$. Therefore, in order to make sure
that all pairs ${\cal P}_{r'}$, ${\cal P}_r$ will fail the second condition, it is enough to place the traps only on very few partitions ${\cal P}_b$, where
by few we mean that there will be a Tower-type jump between their indices.

So with one serious caveat, if one wants to construct an $(\epsilon,f)$-regular partition by taking ${\cal A}$ and ${\cal B}$ to be two of the canonical partitions ${\cal P}_r$,${\cal P}_{r'}$, then one is forced to take
two partitions that refine the last trap we have placed in $H$. The reason is that by property (1) the integers $r$ and $r'$ must be very far apart, and the way we have placed the traps will guarantee that there will be a trap in between them which will then make sure that they do not satisfy the second property of Definition \ref{def:reg}. The caveat we are referring to is the fact that once we have
added the traps to $G$, we have destroyed the critical feature of the graph $G$, which is that the weights decrease exponentially (recall the observation we made above and the discussion at the end of Subsection \ref{subsec:trap}). Hence, it is no longer true that once we find a discrepancy in some
partition ${\cal P}_r$, this discrepancy cannot be canceled by lower levels. In terms of analyzing Gowers' example, it might be the case that some pairs which were not $\gamma$-regular in $G$, might become
$\gamma$-regular in $H$. Actually, there {\em will} be such pairs. This might completely ruin our ability to prove the $H$ has only $\gamma$-regular partitions of Tower-size.

We overcome the above problem by proving that it cannot happen {\em very often}. Namely, since the trap we have added originates from a random graph, then at least on average we expect it to contribute the same
density to all pairs of vertex sets. So {\em on average}, we do not expect a trap to cancel a discrepancy caused by partitions that are refined by it. This is of course only true on average. To turn this
into a deterministic statement, we formulate a condition that holds in random graphs, and show that if too many pairs that were supposed to be not $\gamma$-regular somehow turn out to be $\gamma$-regular, then
we get a violation of the property we assume the trap to satisfy. Turning this intuition into formality is probably the most challenging part of this paper.
One of the main reasons is that we cannot run this argument over all the pairs; instead we need to somehow ``pack'' them together and then argue
about each of these packaged pairs. See Lemmas \ref{clm:oneset} and \ref{clm:key}.

We now turn to the key lemmas of the paper. To state them we will need to define the notion of \emph{$\beta$-refinement}. We briefly mention that this notion is crucial in overcoming
another assumption we have used in the above discussion, that one is trying to construct an $(\epsilon,f)$-regular partition by using only the canonical partitions ${\cal P}_r$.
Using the notion of $\beta$-refinement we will show that one actually has to {\em approximately} use only such partitions.

Let $0 \leq \beta <1/2$. Given two sets $Z$ and $X$,
we write $Z \subset_{\beta} X$, to denote the fact that $|Z \cap X| \geq (1- \beta) |Z|$.
We will sometimes also say that $X$ $\beta$-contains $Z$ or that $Z$ is $\beta$-contained in $X$ to refer to the fact that $Z \subset_{\beta} X$.
Note that since we assume that $\beta < 1/2$, there can be at most one set $X$ which $\beta$-contains a set $Z$.
Given two partitions ${\cal P}=\{X_1, \ldots, X_m\}$
and ${\cal Z}=\{Z_1, \ldots, Z_k\}$ of $V(H)$ and $0 \leq \beta < 1/2$, we shall say that ${\cal Z}$ is a $\beta$-refinement of ${\cal P}$
if for at least $(1-\beta)k$ values of $t$, there exists $i$ such that $Z_t \subset_{\beta} X_i$.
Observe that if $\beta = 0$, then $\beta$-refinement coincides with the standard notion
of one partition refining another one, that we discussed earlier.


In what follows, when we refer to the graph $H$ we mean the graph $H$ defined in the previous section.
We now state the two key lemmas we will prove later on in the paper. Getting back to the intuitive discussion above, one can think of the first lemma as formalizing condition (1) mentioned above, which we wanted $H$ to satisfy.

\medskip

\begin{lemma}\label{lem:gowtrap} Let $f(x)=1/x$. Suppose ${\cal A}$ and ${\cal B}$ form an $(\epsilon,f)$-regular partition of $H$.
If $|{\cal A}|=k \geq 1/\epsilon$ then ${\cal B}$ is an $\epsilon^{1/5}$-refinement of ${\cal P}_{2\log\log k}$.
\end{lemma}

Note that if $\beta <1/2$ and partition ${\cal A}$ is a $\beta$-refinement of ${\cal P}_r$ then the order of ${\cal A}$ is at least
half
the order of ${\cal P}_r$.
Hence the above lemma (implicitly) says  that partition ${\cal B}$, which must be $1/k$-regular, must have order at least half times the order of ${\cal P}_{2\log\log k}$. Recalling (\ref{eq:parorder}), this means
that $|{\cal B}|\geq (1/2) \cdot T^{\phi}(2 \log\log k)$. We note however, that knowing that ${\cal B}$ must have Tower size is not enough for our proof to work. We actually need to know that
${\cal B}$ is a good refinement of partition ${\cal P}_{2\log\log k}$. This is needed in order to show that if a trap was placed between ${\cal A}$ and ${\cal B}$ then they will indeed fail
to satisfy the second property of Definition \ref{def:reg}. This is exactly where the notion of $\beta$-refinement becomes useful, as we state in the second key lemma, which formalizes property (2) mentioned above that we wanted $H$ to satisfy.

\begin{lemma}\label{lem:trapworks}
Suppose ${\cal A}$, ${\cal B}$ are two partitions of $H$ with the following properties
\begin{itemize}
\item ${\cal B}$ is a refinement of ${\cal A}$.
\item $|{\cal A}|=k$ and $H$ has a trap on a canonical partition ${\cal P}_{b}$ whose order is at least $k^2$.
\item ${\cal B}$ is an $\epsilon^{1/5}$-refinement of ${\cal P}_b$.
\end{itemize}
Then ${\cal A}$ and ${\cal B}$ do not satisfy the second condition of Definition \ref{def:reg}. In particular
they do not form an $(\epsilon,f)$-regular partition of $H$.
\end{lemma}

\medskip

We end this section with the derivation of Theorem \ref{thm:main} from Lemma \ref{lem:gowtrap} and Lemma \ref{lem:trapworks}.

\begin{proof}[Proof of Theorem \ref{thm:main}] Suppose ${\cal A}$ and ${\cal B}$ form an $(\epsilon,f)$-regular partition of $H$,
where $|{\cal A}|=k\geq 1/\epsilon$. Let $m_s$ denote the order of ${\cal P}_s$, which is the largest partition on which we have placed a trap.
Recall that $s=w(\frac{1}{48}\sqrt{\log(1/\epsilon)})$ and that $m_s \geq s$ (in fact, $m_s=T^{\phi}(s)$). Hence, by Claim \ref{clm:W} we have
$m_s \geq W(\frac{1}{96}\sqrt{\log(1/\epsilon)})$. Therefore, if $k \geq \sqrt{m_s}$ we are done since
$\sqrt{W(\frac{1}{96}\sqrt{\log(1/\epsilon)})} > W(\frac{1}{100}\sqrt{\log(1/\epsilon)})$ (with a lot of room to spare).

We can thus assume that $|{\cal A}|=k \leq \sqrt{m_s}$, and choose $b$ to be the {\em smallest} index of a partition ${\cal P}_{b}$, on which we have placed a trap
satisfying $|{\cal P}_b| \geq k^2$. If we could show that ${\cal B}$ forms an $\epsilon^{1/5}$-refinement of ${\cal P}_{b}$, then an application
of Lemma \ref{lem:trapworks} would give that ${\cal A}$ and ${\cal B}$ do not form an $(\epsilon,f)$-regular partition of $H$, which would be a
contradiction. Now, Lemma \ref{lem:gowtrap} tells us that ${\cal B}$ is an $\epsilon^{1/5}$-refinement of ${\cal P}_{2\log\log k}$.
Note that if ${\cal B}$ is an $\epsilon^{1/5}$-refinement on ${\cal P}_{2\log\log k}$ then it is also an $\epsilon^{1/5}$-refinement of any partition
that is refined by ${\cal P}_{2\log\log k}$. In other words, it is enough\footnote{Recall that each partition
${\cal P}_r$ is a refinement of all the partitions ${\cal P}_{r'}$ with $r' \leq r$.} that we show that $b \leq 2\log\log(k)$.

Suppose first that $b=w(1)$, that is, the first trap of size at least $k^2$ is the first trap placed in $H$. Then recalling (\ref{eqw0}) and the fact that
$k \geq 1/\epsilon$, we have
$$
b=w(1)=\lfloor \log\log(1/\epsilon)\rfloor \leq 2\log\log(k)\;,
$$
as needed. Suppose now that $b=w(g+1)$ for some $g \geq 1$ and that the trap with largest order smaller than $k^2$ was placed on ${\cal P}_{b'}$ where $b'=w(g)$. Then recalling (\ref{eqw}) we
see that $b=\lfloor \log\log (T^{\phi}(b'))\rfloor$. We also recall (\ref{eq:parorder}) stating that $|{\cal P}_{b'}|=T^{\phi}(b')$. We thus infer that
$$
T^{\phi}(b')= |{\cal P}_{b'}| \leq k^2\;,
$$
implying that
$$
b= \lfloor \log\log(T^{\phi}(b'))\rfloor \leq \log\log(k^2) \leq 2\log\log(k)\;,
$$
thus completing the proof. \end{proof}

As one can see from our proof of Theorem \ref{thm:main},
what we show is not only that an $(\epsilon,f)$-regular partition must be large, but that the only way to get such a partition
is to basically take ${\cal A}$ and ${\cal B}$ to be refinements of partition ${\cal P}_s$ in $H$.
Recall that we started this section by saying that one should design $H$ in a way that will make sure that
at least the proof of Theorem \ref{thm:afks} will produce a large partition. The fact that the only way to get an
$(\epsilon,f)$-regular partition is to take partition ${\cal P}_s$, can be interpreted as saying that the only way
to {\em prove} Theorem \ref{thm:afks} is to go through the process described at the beginning of this section.

\section{Some Preliminary Lemmas}\label{sec:trap}

In this section we prove some simple lemmas that will be used later on in the paper. But we start with proving the claims that were stated without proof in the previous sections. From this point on, when we write something like $x \leq_{(\ref{eqw0})} y$, we mean that the fact that $x \leq y$ follows from the facts stated in equation (\ref{eqw0}). As the reader will inevitably notice, we will be very loose in many of the proofs. The main reason is that as we are
dealing with $W$-type and Tower-type functions, many ``improvements'' make
absolutely no difference even on the quantitative bounds one obtains. Hence, we opted
for statements that are simpler to state and apply.

\begin{proof}[Proof of Claim \ref{lem:balanced}]
First, notice that for any $m \geq 1$, we
can simply repeat the partition
$A_i = \{1\}, B_i = \{2\}$, a total of $m$ times to get $m$ partitions of the set $\{1, 2\}$
such that  there is no $i$ for which (distinct) $j,j'$
appear in the same part.
Since for $1 \leq m \leq 16$, we have $\phi(m) = 2^{\lceil m/16 \rceil} = 2$, the claim holds
for these values of $m$.

Suppose now that $m \geq 17$, set $M = \phi(m) = 2^{\lceil m/16 \rceil}$ and
consider a randomly generated sequence $(A_i,B_i)^{m}_{i=1}$ of partitions of $[M]$ obtained as follows;
for each $1 \leq i \leq m$ and each $1 \leq j \leq M$ we assign element $j$ to $A_i$ with probability $1/2$ (all $mM$ choices being independent).
Fix a pair of distinct elements $j,j' \in [M]$. Clearly the number of $i$ such that $j,j'$ belong to the same class in $(A_i,B_i)$ is
distributed as the binomial random variable $B(m,1/2)$. Hence, we get from a standard application of Chernoff's inequality that the probability that the number of
these $i$ is larger than $3m/4$ is bounded by $e^{-m/6}$. Hence, the probability that some pair of distinct $j,j' \in [M]$ belong to the same part in more than
$3m/4$ of the partitions is bounded by ${M \choose 2}e^{-m/6} < 1$ so the required sequence of partitions exists.
\end{proof}

\begin{proof}[Proof of Claim \ref{clm:W}]
Let us start by proving that

\begin{equation}\label{eqtower1}
 T^{\phi}(x) \geq T(\lfloor x/2 \rfloor) \;,
\end{equation}
as we have previously claimed. We first notice that when $x \geq 256$ we have
$2^{x/16} \geq 16 x$, implying that in this case we have

\begin{equation}\label{eqtower2}
\phi(\phi(t)) \geq 2^{2^{t/16}/16} \geq 2^t\;.
\end{equation}
Now, one can verify that (\ref{eqtower1}) holds when $1 \leq x \leq 10$ and that $T(x)\geq 256$
when $x \geq 4$. Thus, when $x \geq 11$, we have

$$
T^{\phi}(x) \geq \phi(\phi(T^{\phi}(x-2))) \geq_{(\ref{eqtower1})} \phi(\phi(T(\lfloor x/2 \rfloor - 1))) \geq_{(\ref{eqtower2})} 2^{T(\lfloor x/2 \rfloor - 1)} = T(\lfloor x/2 \rfloor)\; .
$$

We now recall (\ref{eqw0}) which implies that since we can assume that $\epsilon$ is small enough, we can also assume that $w(1)$ is large enough.
In particular we have $w(1) \gg W(1) = T(1) = 2$. Let us denote $\hat{T}(t) = \lfloor \log \log (T^{\phi}(t)) \rfloor$. So $w(i)$ is just $\hat T$ iterated $i$ times with $w(1) = \lfloor \log \log (1/\epsilon) \rfloor$. Now we shall show that for any large enough $t$, $\hat T(\hat T(t)) > T(t)$. Using induction, it would follow that for all $i \geq 1$, $w(i) >  W(\lfloor i/2 \rfloor)$, thus completing the proof. Now
\begin{eqnarray*}
\hat T(\hat T(t)) &=& \lfloor \log \log (T^{\phi}(\lfloor \log \log (T^{\phi}(t) )\rfloor ) ) \rfloor \\
& \geq &  \frac14\log \log \left(T \left(\frac{1}{4}  \log \log \left(T \left(t/4 \right) \right)\right) \right)  \\
& \geq & \frac14T \left(\frac{1}{4} T \left(t/4 - 2  \right) - 2 \right) \\
& \geq & \frac14T \left(\frac{1}{5} T\left( \frac{t}{5}\right) \right) \\
& \geq & T(t)\;,
\end{eqnarray*}
where in the first inequality we apply (\ref{eqtower1}), in the second we use the fact that $\log \log (T(x)) = T(x-2)$,
and the last holds for all large enough $t$.
\end{proof}

We now turn to the proof of Claim \ref{trapsexist}.
Recall that given two sets of vertices $R,R'$, which are not necessarily disjoint, we used $e(R,R')$ to denote the number
of edges connecting a vertex in $R$ to a vertex in $R'$, where an edge belonging to $R \cap R'$ is counted twice.

\begin{claim}\label{lem:random1}
There is a constant $C$, such that if $m=m_b \geq C$ and ${\cal O}$ is a random graph from $G(m, 1/2)$, then with probability at least
$3/4$ it satisfies the first condition of a trap (as stated in Definition \ref{def:trap}).
\end{claim}


\begin{proof}
Fix two sets $R,R'$ of size $r=\lceil \sqrt{m}/4 \rceil$. Given distinct $\ell,\ell'$ let $z_{\ell,\ell'}$ be the indicator for the event
that $(\ell,\ell') \in E({\cal O})$, and $z_{R,R'}=\sum_{\ell \in R, \ell' \in R'}z_{\ell,\ell'}$. Then,
$$
\frac{3r^2}{8} \leq {r \choose 2} \leq {\mathbb E}[z_{R,R'}]={\mathbb E}[e(R,R')]=\frac12\left(r^2-|R \cap R'|\right) \leq \frac{r^2}{2}\;,
$$
for all large enough $m$. Now observe that $z_{R,R'}$ is a sum of at least ${r \choose 2}$ indicators $z_{\ell,\ell'}$ and each $z_{\ell,\ell'}$ can change the value of $z_{R,R'}$ by at most $2$.
We thus get from a standard application of Chernoff's inequality that
$$
{\mathbb P}\left[\left|e(R,R')-\frac12r^2\right| \geq \frac14r^2\right] \leq {\mathbb P}\left[\left|z_{R,R'}- {\mathbb E}[z_{R,R'}]\right| \geq \frac18r^2\right] \leq  e^{-\frac{r^2}{100}}\;.
$$
Hence the probability that there is any pair of sets $R,R'$ satisfying $|e(R,R')-\frac12 r^2| > \frac14r^2$ is at most
$$
{m \choose r}^22^{-\frac{1}{100}r^2} \leq m^{\sqrt{m}}e^{-m/1600} \ll 1/4\;,
$$
for all large enough $m$. \end{proof}

\begin{claim}\label{lem:random2}
There is a constant $C$, such that if $m=m_b \geq C$ and ${\cal O}$ is a random graph from $G(m, 1/2)$, then with probability at least
$3/4$, it satisfies the second condition of a trap (as stated in Definition \ref{def:trap}).
\end{claim}

\begin{proof} Let us start by considering the case $b'=b-1$. Suppose $U_1,\ldots,U_{m_{b-1}}$ is the partition of $V({\cal O})$ induced by the partition
${\cal P}_{b-1}$ (as discussed prior to Definition \ref{def:trap}).
Now recall (see Subsection \ref{subsec:gowconstruct}) that the integers $m_b$ satisfy the relation
$$
m=m_b=m_{b-1}\phi(m_{b-1})=m_{b-1}2^{\lceil m_{b-1}/16 \rceil}\;.
$$
This means that
\begin{equation}\label{eqt0}
\log(m) \leq m_{b-1} \leq 17\log (m)\;,
\end{equation}
so the size of the sets $U_i$, which we will denote by $h_{b-1}$, satisfies
\begin{equation}\label{eqt1}
m/17\log(m) \leq h_{b-1}=m/m_{b-1} \leq m/\log(m)\;.
\end{equation}
Fix now two sets $U_i,U_j$, an integer $200 \leq k \leq \log(m)$, a subset $R \subseteq U_i$ of size $k^6$ and a subset $R' \subseteq U_j$ of size
$\lceil h_{b-1}/k \rceil$.  Given distinct $\ell,\ell'$ with $\ell \in R$ and $\ell' \in R'$ let $z_{\ell,\ell'}$ be the indicator for the event
that $(\ell,\ell') \in E({\cal O})$, and $z_{R,R'}=\sum_{\ell \in R, \ell' \in R'}z_{\ell,\ell'}$. Then
\begin{eqnarray*}
\frac{|R||R'|}{2} \geq {\mathbb E}[z_{R,R'}]={\mathbb E}[e(R,R')]&=&\frac12\left(|R||R'|- |R \cap R'|\right)\\
&\geq& \frac{1}{2}|R||R'|-\frac{1}{2}|R|\\
&\geq& \left(\frac12-\frac{1}{2k^2}\right)|R||R'| \;.
\end{eqnarray*}
where in the last inequality we use the facts that $k \leq \log(m)$, that $|R'|= \lceil h_{b-1}/k \rceil \geq_{(\ref{eqt1})} m/17k \log(m)
\geq m/17 \log^2 (m)$ and that we can pick $m$ to be large enough
so that $|R'| \geq k^2$.

Note that $z_{R,R'}$ is a sum of at least $|R|(|R'|-|R|) \geq |R||R'|/2$ indicators $z_{\ell,\ell'}$ (we are using the fact that $|R| \ll |R'|$).
Since each of them can change $z_{R,R'}$ by at most $2$,
we get from Chernoff's inequality, the fact that $k \geq 200$, and the estimate for ${\mathbb E}[z_{R,R'}]$ from the previous paragraph that
\begin{eqnarray*}
{\mathbb P}\left[\left|e(R,R')-\frac12|R||R'|\right| \geq \frac{1}{k^2}|R||R'|\right] &\leq& {\mathbb P}\left[\left|z_{R,R'}-{\mathbb E}[z_{R,R'}]\right| \geq \frac{1}{2k^2}|R||R'|\right]\\
  &\leq& e^{-\frac{|R||R'|}{64k^4}}\\
  &\leq& e^{-kh_{b-1}/64}\\
  &\leq&  e^{-2h_{b-1}}\;.
\end{eqnarray*}
Now, there are $m^2_{b-1} = O(\log^2(m))$ ways to pick the sets $U_i,U_j$, $O(\log(m))$ ways to choose $k$, ${h_{b-1} \choose k^6}$ ways to pick $R$ and ${h_{b-1} \choose h_{b-1}/k}$ ways to pick
$R'$. Overall, we get from a union bound that the probability that some choice of $U_i$, $U_j$, $k$, $R$ and $R'$ will violate the second condition of Definition \ref{def:trap} is bounded by
\begin{equation}\label{eqt2}
O(\log^3 m){h_{b-1} \choose k^6}{h_{b-1} \choose h_{b-1}/k}e^{-2h_{b-1}} \leq m^{2k^6}(ek)^{h_{b-1}/k}e^{-2h_{b-1}} \leq m^{2\log^6(m)}e^{-h_{b-1}}\;,
\end{equation}
where in the first inequality we use the inequality ${n \choose k} \leq (en/k)^k$ and in the second the fact that $k \leq \log(m)$.

Let us now consider an arbitrary $b' < b$. Note that since $m_{b'} \leq m_{b-1}$, we still have
$m_{b'} \leq 17\log(m)$. Hence there are still only $O(\log^2(m))$ many ways to choose the sets $U^{b'}_i,U^{b'}_j$. This means that the upper
bound obtained in (\ref{eqt2}) for the probability of partition ${\cal P}_{b-1}$ violating the condition applies to any given partition ${\cal P}_{b'}$, with
$h_{b-1}$ replaced by $h_{b'}$. But since $h_{b'} \geq h_{b-1}$ the right hand side of the bound in (\ref{eqt2}) still holds.

We finally recall (\ref{eq:parorder}) stating that $m_b=T^{\phi}(b)$.
As we noted in (\ref{eqtower1}) we have $T^{\phi}(b) > T(\lfloor b/2 \rfloor)$. Hence the number of $b' < b$ we need to consider is only $O(\log^{*}(m))$. So combining this fact
with the discussion in the previous paragraph we get that the probability of any partition ${\cal P}_{b'}$ violating the second condition of Definition \ref{def:trap} is
bounded by
$$
m^{3\log^6(m)}e^{-h_{b-1}} \ll 1/4
$$
where we apply the fact that $h_{b-1} \geq m/17\log(m)$, stated in (\ref{eqt1}). \end{proof}

\begin{proof}[Proof of Claim \ref{trapsexist}] Follows immediately from Claims \ref{lem:random1} and \ref{lem:random2}.
\end{proof}

We will now prove two lemmas which will somewhat streamline the application of the properties of traps later on in the paper.
Both lemmas will rely on the observation stated in Lemma \ref{convcomb} below.
In what follows, we use $v_S \in \mathbb{R}^n$, with $S \subseteq [n]$ to denote the vector whose $i^{th}$ entry is $1/|S|$ when $i \in S$ and $0$ otherwise.
Let ${\cal V}_k=\{v_S:S \subseteq [n], |S|=k\}$.

\begin{lemma}\label{convcomb} If $x \in [0,1/k]^n$ and $\sum x_i=1$, then $x$ is a convex combination of the vectors of ${\cal V}_k$.
\end{lemma}

Before we prove this lemma, we need a standard theorem from linear programming theory,
which we state without proof. A polyhedron $P\subseteq \mathbb R^n$ is the set of points satisfying a finite number
of linear inequalities. $P$ is {\em bounded} if there is a constant $C$ such that $\lVert x \rVert \leq C$ for all $x \in P$.
Finally, a point $x \in P$ is said to be a {\em vertex} of $P$ if it cannot be represented as a proper convex combination
of two distinct points $x', x'' \in P$.

\begin{theorem}[\cite{bertsimas}]\label{thm:lpconvex}
For every bounded polyhedron $P \subseteq \mathbb R^n$ and $x \in P$, the point $x$
can be written as a convex combination of the vertices of $P$.
\end{theorem}

\begin{proof}[Proof of Lemma \ref{convcomb}]
Consider the polyhedron
$$P = \left\{ x  ~: \;\sum_i x_i = 1, \mbox{ and } 0 \leq x_1,\ldots, x_n\leq 1/k\right\}\;.$$
Notice that for all $x \in P$, we have $\lVert x \rVert \leq 1$. Let ${\cal V}$ be the set of vertices of $P$.
By Theorem \ref{thm:lpconvex}, we have that any $x \in P$ is a convex combination of ${\cal V}$. So we need to show that\footnote{We clearly have ${\cal V}_k \subseteq {\cal V}$ but this direction is not needed.}
${\cal V} \subseteq {\cal V}_k$.

Suppose $u \in {\cal V}$. If all its entries are either $0$ or $1/k$ it obviously belongs to ${\cal V}_k$.
So suppose that $u$ has an entry $u_i \in (0, 1/k)$. Then there exists at least one more entry $u_j \in (0, 1/k)$,
because otherwise the entries cannot sum to $1$. Let $\ve_u = \frac{1}{2} \min \{u_i, u_j, 1/k-u_i, 1/k - u_j \}$.
Let $e_i$ denote the canonical basis vector where the $i$th entry is 1 and all the other
entries are 0. Similarly define $e_j$. Let $u' = u + \ve_u e_i - \ve_u e_j$
and $u'' = u - \ve_u e_i + \ve_u e_j$. It can be checked that both $u', u'' \in P$
and that $u' + u'' = 2u$. So $u$ can be written as the convex combination of two other
vectors in $P$, which means that $u$ is not a vertex of $P$.
\end{proof}

We now turn to prove two lemmas. The first one will help us in applying the first property of traps in proving Lemma \ref{lem:trapworks},
while the second one will help us in applying the second property of traps in proving Lemma \ref{lem:gowtrap}.

\begin{lemma}\label{lem:random} Suppose ${\cal O}$ is the graph that was
used when defining the trap on partition ${\cal P}_{b}$ (so $|V({\cal O})|=m_b$ and we can assume that ${\cal O}$ satisfies the first condition of Definition \ref{def:trap}).
Let $Q$ be the adjacency matrix of ${\cal O}$, and suppose $x,y \in [0,1]^{m_b}$ satisfy $\sum x_i=\sum y_i=g \geq \sqrt{m_b}/2$.
Then we have
$$
\left|x^TQy-\frac12g^2\right| \leq \frac14g^2\;.
$$
\end{lemma}


\begin{proof} The vectors $x/g$ and $y/g$ satisfy the condition of Lemma \ref{convcomb} with $k=\lceil \sqrt{m_b}/4 \rceil$. Hence we can express
$x/g$ and $y/g$ as convex combinations of the vectors of ${\cal V}_k$
as $x/g=\sum_{R}a_R v_R$ and $y/g=\sum_{R'}b_{R'} v_{R'}$.
Observe further that $(v_R)^TQv_{R'}=e(R,R')/|R||R'|$.
Since $|R|=|R'|=k=\lceil \sqrt{m_b}/4 \rceil$ and we assume that ${\cal O}$ satisfies the first condition of being a trap, we can infer that for any $R$ and $R'$ we have
\begin{equation}\label{eqexp2}
1/4 \leq (v_R)^TQv_{R'} \leq 3/4\;.
\end{equation}
We can thus infer from (\ref{eqexp2}) and the fact that $\sum_{R}a_{R} v_{R}$ and $\sum_{R'}b_{R'} v_{R'}$ are convex combinations that
\begin{eqnarray*}
(x/g)^TQ(y/g)&=&\left(\sum_{R}a_R v_R\right)^TQ\left(\sum_{R'}b_{R'} v_{R'}\right)\\
     &=&\sum_{R,R'}a_{R}b_{R'} (v_{R})^TQv_{R'}\\
     &\leq& \frac34 \sum_{R,R'}a_{R}b_{R'}\\
     &=&\frac34\;,
\end{eqnarray*}
implying that $x^TQy \leq \frac34g^2 $. An identical argument gives $x^TQy \geq \frac14g^2$, which completes the proof. \end{proof}

\begin{lemma}\label{lem:random3} Suppose ${\cal O}$ is the graph that was
used when defining the trap placed on partition ${\cal P}_{b}$ (so $|V({\cal O})|=m_b$ and we can assume that ${\cal O}$ satisfies the second condition of Definition \ref{def:trap}).
Let $Q$ be the adjacency matrix of ${\cal O}$.
Let $b' < b$, set $m=m_{b'}$ and let $X_1,\ldots,X_{m}$ be the partition of $V({\cal O})$ induced\footnote{This was defined explicitly just before Definition \ref{def:trap}.
Since we are identifying the clusters of ${\cal P}_b$ with the vertices of ${\cal O}$ we can also identify these clusters with the indices of the adjacency matrix $Q$.
Hence, since we think of $X_i$ as a subset of vertices of ${\cal O}$, we can say (as we will in item 2) that an index of a vector $x \in [0,1]^{m_b}$ belongs to $X_i$.} by ${\cal P}_{b'}$.
Suppose each of the sets $X_i$ has size $h$ and let $X_i,X_j$ be two of these sets.
Suppose $\delta$ and $x, y \in [0,1]^{m_b}$ satisfy the following conditions:
\begin{enumerate}
\item  $1/\log(m_b)< \delta < 1/200$.
\item The vector $x$ has non-zero entries only in $X_i$ and $y$ has non-zero entries only in $X_j$.
\item For each $1 \leq p' \leq m_b$ we have $x_{p'}/(\sum_p x_p) < \delta^6$.
\item $\sum^{m_b}_{p=1} y_p > 2\delta h$.
\end{enumerate}
Then, setting $g_1=\sum_p x_p$ and $g_2=\sum_p y_p$, we have
\begin{equation}\label{eqtcond}
\left|x^TQy-\frac12g_1g_2\right| \leq 2\delta^2g_1g_2\;.
\end{equation}
\end{lemma}

\begin{proof} Put $k=\lfloor 1/\delta \rfloor$. Then item (1) of the lemma guarantees that $200 \leq k \leq \log(m_b)$.
Item (3) of the lemma guarantees that the vector $x/g_1$
satisfies the condition of Lemma \ref{convcomb} with respect to $k^6$.
Hence we can write $x/g_1=\sum_{R}a_R v_R$ using the vectors of ${\cal V}_{k^6}$.
Moreover, since item (2) guarantees that $x$ has non-zero entries only in $X_i$ we know
that in the convex combination $\sum_{R}a_R v_R$
all the sets $R$ satisfy $R \subseteq X_i$.
Observe now that item (2) guarantees that $y$ has non-zero entries only in $X_j$.
Item (4) of the lemma guarantees that the vector $y/g_2$
satisfies the condition of Lemma \ref{convcomb} with respect to $\lceil h/k \rceil$.
Hence we can write $y/g_2=\sum_{R'}b_{R'}v_{R'}$ using the vectors of ${\cal V}_{\lceil h/k \rceil}$.
Again, we know that in this convex combination we are only using sets $R' \subseteq X_j$.

Now, $(v_{R})^TQv_{R'}=e(R,R')/|R||R'|$. Hence, if $|R|=k^6$ and $|R'|=\lceil h/k \rceil$
and $R \subseteq X_i$, $R' \subseteq X_j$, then we can use the assumption
that ${\cal O}$ satisfies the second condition of being a trap, to conclude that
\begin{equation}\label{eqexp5}
\left|(v_{R})^TQv_{R'} - \frac12\right| \leq 1/k^2 \leq 2\delta^2\;.
\end{equation}
We can thus infer from (\ref{eqexp5}) and the facts that $\sum_{R}a_R v_R$ and $\sum_{R'}b_{R'} v_{R'}$ are convex combinations that
\begin{eqnarray*}
(x/g_1)^TQ(y/g_2)&=&\left(\sum_{R}a_R v_R\right)^TQ\left(\sum_{R'}b_{R'} v_{R'}\right)\\
     &=&\sum_{R,R'}a_R b_{R'}  (v_R)^TQv_{R'}\\
     &\leq& (1/2+2\delta^2) \sum_{R,R'}a_Rb_{R'}\\
     &=&(1/2+2\delta^2)
\end{eqnarray*}
implying that $x^TQ y \leq (1/2+2\delta^2)g_1g_2$. An identical argument gives $x^TQy \geq (1/2-2\delta^2)g_1g_2$, which completes the proof.
\end{proof}

\section{Proof of Lemma \ref{lem:trapworks}}\label{sec:lem2}

Suppose ${\cal A}=\{V_i: 1 \leq i \leq k\}$ and ${\cal B}=\{U_{i,i'}: 1 \leq i \leq k, 1 \leq i' \leq \ell\}$ (so $|{\cal B}|=k\ell$).
We will say that a pair of sets $(V_i,V_j)$ is {\em bad} if there are two sets $C_1,C_2 \subseteq [\ell] \times [\ell]$, each of size at least
$\epsilon \ell^2$ such that $|d(U_{i,i_1},U_{j,j_1})-d(U_{i,i_2},U_{j,j_2})| \geq 2\epsilon$ for every $(i_1,j_1) \in C_1$ and $(i_2,j_2) \in C_2$. Note that if $(V_i,V_j)$ is
bad then it cannot be good in the sense of Definition \ref{def:reg}. Hence, to show
that ${\cal A}$ and ${\cal B}$ fail to satisfy the second condition of Definition \ref{def:reg} it is enough to show that there are at least $\epsilon {k \choose 2}$ bad pairs $(V_i,V_j)$.
As we mentioned after the statement of Theorem \ref{thm:main}, we will actually show that there at least $(1-2\epsilon^{1/10}){k \choose 2}$ bad pairs.

A set $U_{i,i'}$ is called {\em useful} if there is an $X \in {\cal P}_b$ such that
$U_{i,i'} \subset_{\epsilon^{1/5}} X$.
If $U_{i,i'}$ is not useful, we call it {\em useless}.
A set $V_i$ is called {\em useful} if it contains\footnote{Recall that each $V_i$ is the union of $\ell$ sets $U_{i,i'}$.} less than 
$\epsilon^{1/10} \ell$
useless sets $U_{i,i'}$.
If $V_{i}$ is not useful, we call it {\em useless}.
Observe that there can be at most
$\epsilon^{1/10}k$
useless sets $V_i$, as otherwise ${\cal B}$
would not be an
$\epsilon^{1/5}$-refinement of ${\cal P}_b$, which would contradict the third assumption of the lemma.
Hence, there are at least $(1-2\epsilon^{1/10}){k \choose 2}$ pairs of useful sets $(V_i,V_j)$.
By the previous paragraph, it is enough to show that every such pair is bad.

So for the rest of the proof, let us fix a pair of useful sets $(V_i,V_j)$.
Let us assume that $\epsilon$ is small enough so that
$\epsilon^{1/5} < 1/2$.
Given a useful set $U_{i,i'} \subset_{\epsilon^{1/5}} X \in {\cal P}_b$,
we let $X_{{\cal P}_b}(U_{i,i'})$ denote this (unique) cluster in ${\cal P}_b$ that $\epsilon^{1/5}$-contains $U_{i,i'}$. We will later prove the following claim:

\begin{claim}\label{clm:middle} If $~V_i$ and $V_j$ are both useful, then there are $D_1,D_2 \subseteq [\ell] \times [\ell]$ satisfying the following:
\begin{itemize}
\item $D_1$ and $D_2$ have size at least $\frac{1}{32} \ell^2$.
\item For every $(i_1,j_1) \in D_1$ both $U_{i,i_1}$ and $U_{j,j_1}$ are useful and the pair $(X_{{\cal P}_b}(U_{i,i_1}),X_{{\cal P}_b}(U_{j,j_1}))$ belongs to the trap placed on ${\cal P}_b$.
\item For every $(i_2,j_2) \in D_2$ both $U_{i,i_2}$ and $U_{j,j_2}$ are useful and the pair $(X_{{\cal P}_b}(U_{i,i_2}),X_{{\cal P}_b}(U_{j,j_2}))$ does not belong to the trap placed on ${\cal P}_b$.
\end{itemize}
\end{claim}

In the next subsection we prove the lemma assuming Claim \ref{clm:middle}, in the subsection following it we will prove this claim.

\subsection{Proof of Lemma \ref{lem:trapworks} via Claim \ref{clm:middle}}

Let $\alpha$ be the weight added to $H$ by the trap that was placed on ${\cal P}_b$.
Let $D_1,D_2$ be the subsets of $[\ell] \times [\ell]$ guaranteed by Claim \ref{clm:middle}.
Take any pair $(i_1,j_1) \in D_1$ and let $X_1=X_{{\cal P}_b}(U_{i,i_1})$ and $X_2=X_{{\cal P}_b}(U_{j,j_1})$. Since $(i_1,j_1) \in D_1$ we know that the pair $(X_1,X_2)$ was
assigned an extra weight of $\alpha$ by the trap placed on ${\cal P}_b$.
Now consider the traps with weight larger than $\alpha$, that is, the traps
that were placed on partitions ${\cal P'}$ which are refined by ${\cal P}_b$.
Note that $(X_1,X_2)$ might get an extra weight from a subset of these traps\footnote{More precisely, if $X_1$ and $X_2$ are subsets of the same cluster $X' \in {\cal P}'$,
then they will never get an extra weight from the trap placed on ${\cal P}'$. If they belong to different clusters $X'_1,X'_2 \in {\cal P}'$, then they
will receive an extra weight only if $(X'_1,X'_2)$ belong to the trap placed on ${\cal P}'$.}. But since $H$ contains only $\frac{1}{48}\sqrt{\log (1/\epsilon)}$ many traps, the number of ways to choose the subset of the traps
with weight larger than $\alpha$ from which $(X_1,X_2)$ get an extra weight is bounded by
$2^{\frac{1}{48}\sqrt{\log (1/\epsilon)}} \ll \frac{1}{32\epsilon}$. Hence $D_1$ must have a subset
of pairs of size at least $\epsilon \ell^2$,
denoted $D'_1$, and set of weights $W_1$ (all larger than $\alpha$) with the following property; if
$\alpha' > \alpha$ and ${\cal P'}$ is the partition on which the trap with weight $\alpha'$ was placed then
for any $(i_1,j_1) \in D'_1$ the pair $(X_{\cal P'}(U_{i,i_1}),X_{\cal P'}(U_{j,j_1}))$ belongs to the trap on ${\cal P'}$ if and only if $\alpha' \in W_1$.
We can also define $D'_2$ and $W_2$ in the same manner.

We now claim that we can take $C_1$ and $C_2$ (the sets showing that $(V_i,V_j)$ is bad) to be the sets $D'_1$ and $D'_2$. First, as noted above,
both $D'_1$ and $D'_2$ have size at least $\epsilon \ell^2$. So to finish the proof we will have to show that for every $(i_1,j_1) \in D'_1$ and $(i_2,j_2) \in D'_2$ we have
\begin{equation}\label{eqbad}
|d(U_{i,i_1},U_{j,j_1})-d(U_{i,i_2},U_{j,j_2})| \geq 2\epsilon\;.
\end{equation}
Let $\alpha'$ be the largest weight
that belongs to exactly one of the sets $W_1$ and $W_2$.
Assume without loss of generality that $\alpha' \in W_1$ and $\alpha' \not \in W_2$.
If there is no such weight (that is, $W_1=W_2$) then set $\alpha'=\alpha$. We now recall Fact \ref{trapdensity} which tells us that
\begin{equation}\label{eqalpha}
\alpha' \geq 4^{-\frac{1}{48}\sqrt{\log(1/\epsilon)}}\;.
\end{equation}
Let ${\cal P'}$ be the partition on which the trap with weight $\alpha'$ was placed.
Since traps with weight at least $\alpha$ are placed on partitions that are refined by ${\cal P}_b$, we see that if a set $U_{i,i'}$ is useful with respect to ${\cal P}_b$ it must also be useful
with respect to ${\cal P'}$. This means that for each pair $(i_1,j_1) \in D'_1$ the trap at ${\cal P'}$ increases $d(U_{i,i_1},U_{j,j_1})$ by at least
$$
\alpha'\left(1-\epsilon^{1/5}\right)^2 \geq \alpha'(1-2\epsilon^{1/5}) \geq 0.99\alpha'\;.
$$
Similarly, for each pair $(i_2,j_2) \in D'_2$ the trap at ${\cal P'}$ increases $d(U_{i,i_2},U_{j,j_2})$ by at most
$$
2\alpha' \epsilon^{1/5}  \leq 0.01\alpha'\;.
$$
Hence, disregarding for a moment all the other weights that can be assigned
to these sets in $H$, we see that all the pairs in $(i_1,j_1) \in D'_1$ are such that $d(U_{i,i_1},U_{j,j_1}) \geq 0.99\alpha'$ while all
$(i_2,j_2) \in D'_2$ are such that $d(U_{i,i_2},U_{j,j_2}) \leq 0.01\alpha'$. We will now show that this discrepancy is (essentially) maintained even when considering the entire graph $H$.

First, recall that by Fact \ref{ob1} the total weight assigned to any pair of vertices of $H$ in the graph $G$ is bounded by $1/4^{\sqrt{\log(1/\epsilon)}}$.
Hence, recalling (\ref{eqalpha}), we see that even after taking into
account these weights, we have $d(U_{i,i_2},U_{j,j_2}) \leq 0.02\alpha'$ for any $(i_2,j_2) \in D'_2$. Let us now consider the contribution of the weights coming from traps that were assigned a weight
smaller than $\alpha'$. Since these weights are $\alpha'/4,\alpha'/16,...$ their sum is bounded by $\alpha'/3$, so after taking these weights into account
we still have $d(U_{i,i_2},U_{j,j_2}) \leq 0.36\alpha'$ for any $(i_2,j_2) \in D'_2$. Let us now consider the contribution coming
from traps with weight more than $\alpha'$.
Consider any trap with weight $\alpha'' > \alpha'$ that was placed on a partition ${\cal P''}$.
Recall that by definition of $W_1$, $W_2$ and by our choice of $\alpha'$, either the extra weight $\alpha''$ was added
to all pairs $(X_{\cal P''}(U_{i,i'}),X_{\cal P''}(U_{j,j'}))$ with $(i',j') \in D'_1 \cup D'_2$ or to none of them.
Since all the sets $U_{i,i_1}$ and $U_{j,j_1}$ are useful we see that for each pair $(i_1,j_1) \in D'_1$ the pair
$(U_{i,i_1},U_{j,j_1})$ gets from the trap at ${\cal P''}$ a total weight at least
$$
\alpha''\left(1-\epsilon^{1/5}\right)^2 \geq \alpha''(1-2\epsilon^{1/5})\;.
$$
Set $w$ to be the sum of the weights in $W_1$ that are larger than $\alpha'$. Then the above discussion implies that for each $(i_1,j_1) \in D'_1$ we have
\begin{equation}\label{eqdisc1}
d(U_{i,i_1},U_{j,j_1}) \geq (1-2\epsilon^{1/5})w+0.99\alpha' \geq w+0.99\alpha'-2\epsilon^{1/5}\;.
\end{equation}
Consider now a pair $(i_2,j_2) \in D'_2$; If a weight $\alpha'' \geq \alpha'$ belongs to $W_2$ then it can contribute to $d(U_{i,i_2},U_{j,j_2})$ a weight of at most $\alpha''$, hence
such weights contribute to $d(U_{i,i_2},U_{j,j_2})$ a total weight of at most\footnote{Recall that by choice of $\alpha'$ the sets $W_1$ and $W_2$ contain the same weights larger than $\alpha'$.} $w$.
As to weights $\alpha'' > \alpha'$ that do not belong to $W_2$, we see that since $U_{i,i_2}$ and $U_{j,j_2}$ are useful, they can increase $d(U_{i,i_2},U_{j,j_2})$ by at most
$2\alpha''\epsilon^{1/5}$. As the total sum of weights of all traps is at most 1, this extra contribution is bounded by $2\epsilon^{1/5}$.
All together, we see that for every $(i_2,j_2) \in D'_2$,
\begin{equation}\label{eqdisc2}
d(U_{i,i_2},U_{j,j_2}) \leq w+ 0.36 \alpha' + 2\epsilon^{1/5}.
\end{equation}
Recalling (\ref{eqalpha}), we see that $4\epsilon^{1/5} < 0.1\alpha'$. Hence, (\ref{eqdisc1}) and (\ref{eqdisc2}) imply that
$$
d(U_{i,i_1},U_{j,j_1}) - d(U_{i,i_2},U_{j,j_2}) > 0.5 \alpha' >_{(\ref{eqalpha})} 2\epsilon
$$
for every choice of $(i_1,j_1) \in D'_1$ and $(i_2,j_2) \in D'_2$. This establishes (\ref{eqbad}), thus completing the proof.

\subsection{Proof of Claim \ref{clm:middle}}

Let us start with observing that since $V_i$ is assumed to be useful, it contains (more than) $\frac12\ell$ useful sets $U_{i,i'}$.
Let $V'_i$ be the union of $\frac12\ell$ such sets, and define $V'_j$ is a similar way.
From now on we will focus on $V'_i$ and $V'_j$ and their subsets $U_{i,i'}$ and $U_{j,j'}$ so we will only be talking about
sets $U_{i,i'}$ and $U_{j,j'}$ that are useful.
Recall that for any useful set $U_{i,i'}$ there is a (unique) set $X_{{\cal P}_b}(U_{i,i'}) \in {\cal P}_b$ such that $U_{i,i'} \subset_{\epsilon^{1/5}} X_{{\cal P}_b}(U_{i,i'})$.

Suppose ${\cal P}_b$ has $m$ clusters and recall that we defined the trap on ${\cal P}_b$ using an $m$-vertex graph ${\cal O}$ satisfying the first condition of Definition \ref{def:trap}.
That is $(u,v)$ is an edge of ${\cal O}$ if and only if $(X_u,X_v)$ belongs to the trap on ${\cal P}_b$.
Define a vector $x \in [0,1]^{m}$ by setting $x_u=|V'_i \cap X_u|/|X_u|$. Define $y \in [0,1]^m$ similarly by setting
$y_u=|V'_j \cap X_u|/|X_u|$. Recall that each of the sets $V_i$ contains a $1/k$-fraction of the vertices $H$ (since $|{\cal A}|=k$)
so $|V'_i|$ contains a $1/2k$-fraction of the vertices of $H$. Since ${\cal P}_b$ has order $m$ (so there are $m$ sets $X_u$)
and we assume that $m \geq k^2$ (in the second item of Lemma \ref{lem:trapworks}) we infer that
\begin{equation}\label{eqg}
\sum_ux_u=\sum_uy_u=\frac{m}{2k} \geq \sqrt{m}/2 \;.
\end{equation}
If we take $Q$ to be the adjacency matrix of ${\cal O}$, then by (\ref{eqg}) we can apply Lemma \ref{lem:random} (with $g=m/2k$) to infer that
\begin{equation}\label{eq4}
\frac14(m/2k)^2 \leq x^T Q y \leq \frac34(m/2k)^2\;.
\end{equation}
Given a set $U_{i,i'}$ we define a vector $x^{i'}$ by setting $x^{i'}_u=|U_{i,i'}\cap X_u|/|X_u|$. Similarly given
a set $U_{j,j'}$ we define a vector $y^{j'}$ by setting $y^{j'}_u=|U_{j,j'}\cap X_u|/|X_u|$. Observe that since
$V'_i$ is the union of the sets $U_{i,i'}$ we have $x=\sum_{i'}x^{i'}$ where
the sum ranges over all the $\ell/2$ indices $i'$ for which $U_{i,i'} \subseteq V'_i$. Similarly $y=\sum_{j'}y^{j'}$ where
the sum ranges over all the $\ell/2$ indices $j'$ for which $U_{j,j'} \subseteq V'_j$. Hence, we get from (\ref{eq4}) that
\begin{equation}\label{eq5}
\frac14(m/2k)^2 \leq \sum_{i',j'} (x^{i'})^TQy^{j'} \leq \frac34(m/2k)^2\;.
\end{equation}

Consider now any pair $i',j'$ in the above sum. Let $X_{u'}=X_{{\cal P}_b}(U_{i,i'})$ and $X_{v'}=X_{{\cal P}_b}(U_{j,j'})$. Recall that $U_{i,i'}$ contains a $1/k\ell$ fraction of $V(H)$ while the sets $X_u$ contains a $1/m$ fraction of $V(H)$.
This means that
$$
\sum_ux^{i'}_u=m/k\ell\;,
$$
and similarly we have
$$
\sum_uy^{j'}_u=m/k\ell\;.
$$
Hence
\begin{equation}\label{eq6}
0 \leq (x^{i'})^TQy^{j'} \leq m^2/k^2\ell^2\;.
\end{equation}
More importantly, since $|U_{i,i'} \cap X_{u'}| \geq \left(1-\epsilon^{1/5}\right)|U_{i,i'}|$ we have
\begin{equation}\label{eq7}
x^{i'}_{u'}=|U_{i,i'} \cap X_{u'}|/|X_{u'}| \geq \left(1-\epsilon^{1/5}\right)m/k\ell\;,
\end{equation}
and since $|U_{j,j'} \cap X_{v'}| \geq \left(1-\epsilon^{1/5}\right)|U_{j,j'}|$ we have
\begin{equation}\label{eq8}
y^{j'}_{v'}=|U_{j,j'} \cap X_{v'} |/|X_{v'}| \geq \left(1-\epsilon^{1/5}\right)m/k\ell\;.
\end{equation}

Suppose now that $(X_{u'},X_{v'})$ belong to the trap placed on ${\cal P}_b$, that is, that $Q_{u',v'}=1$.
We then get from (\ref{eq6}), (\ref{eq7}) and (\ref{eq8}) that
\begin{equation}\label{eq9}
0.99m^2/k^2\ell^2 \leq \left(1-\epsilon^{1/5}\right)^2 m^2/k^2\ell^2 \leq (x^{i'})^TQy^{j'} \leq m^2/k^2\ell^2\;.
\end{equation}
Suppose now that $(X_{u'},X_{v'})$ does not belong to the trap placed on ${\cal P}_b$, that is, that $Q_{u',v'}=0$.
We then get from (\ref{eq6}), (\ref{eq7}) and (\ref{eq8}) that
\begin{equation}\label{eq10}
0 \leq (x^{i'})^TQy^{j'} \leq 2\epsilon^{1/5}m^2/k^2\ell^2 \leq 0.01 m^2/k^2\ell^2\;.
\end{equation}
We thus see from (\ref{eq10}) that the total to contribution to (\ref{eq5}) of pairs $(i',j')$ for which $(X_{u'},X_{v'})$ does not belong to the trap is bounded
by $(\ell/2)^2 \cdot 0.01 m^2/k^2\ell^2 = 0.01(m/2k)^2$. Combining (\ref{eq5}), (\ref{eq9}) and (\ref{eq10}) it thus must be the case that there
are at least
$$
\frac{\frac14(m/2k)^2-0.01(m/2k)^2}{m^2/k^2\ell^2} \geq \frac{1}{32}\ell^2\;,
$$
pairs $(i',j')$ for which $(X_{u'},X_{v'})$ belongs to the trap placed on ${\cal P}_b$. Hence we can take $D_1$ to be the collection of these pairs.
Finally, we see from (\ref{eq5}), (\ref{eq9}) and (\ref{eq10}) that the number of pairs $(i',j')$ for which $(X_{u'},X_{v'})$ belongs to the trap on ${\cal P}_b$
cannot be larger than
$$
\frac{\frac34(m/2k)^2}{0.99m^2/k^2\ell^2} \leq \frac{31}{32}\ell^2\;,
$$
so we can take $D_2$ to be the collection of pairs $(i',j')$ that do not belong to $D_1$.
We thus complete the proof of Claim \ref{clm:middle}.

\section{Proof of Lemma \ref{lem:gowtrap}}\label{sec:lem1}

We will prove Lemma \ref{lem:gowtrap} by first performing a series of reductions that will culminate
in Lemma \ref{clm:key}. We will then spend most of this section proving Lemma \ref{clm:key}.
Let us first derive Lemma \ref{lem:gowtrap} from the following lemma:

\begin{lemma}\label{lem:maintrap}
Suppose $\gamma \leq \epsilon$ and ${\cal Z}=\{Z_1, \ldots, Z_k\}$ is a $\gamma$-regular partition of $H$.
Assume
\begin{itemize}
\item $r < \frac{\log (1/\gamma)}{10\sqrt{\log(1/\epsilon)}}$
\item $\gamma^{1/4} \leq \beta \leq 1/100$
\end{itemize}
Then, if ${\cal Z}$ is a $\beta$-refinement of ${\cal P}_{r-1}$ it is also an $8\beta$-refinement of ${\cal P}_{r}$.
\end{lemma}

\begin{proof}[Proof that Lemma \ref{lem:maintrap} implies Lemma \ref{lem:gowtrap}]
By the definition of $(\epsilon,f)$-regularity, we get that if $|{\cal A}|=k$ then ${\cal B}$ must be $\frac{1}{k}$-regular. Since
$k \geq 1/\epsilon$ we have $1/k \leq \epsilon$. Since ${\cal B}$ is a refinement of
${\cal P}_0$ (recall that ${\cal P}_0$ is just the entire vertex set of $H$), it is in particular a $(1/k)^{1/4}$-refinement of ${\cal P}_0$.
Hence, starting with $\beta=(1/k)^{1/4}$ we can repeatedly apply Lemma \ref{lem:maintrap} (with $\gamma=1/k$) as long as
\begin{equation}\label{eqp1}
r \leq \frac{\sqrt{\log(k)}}{10} \leq \frac{\log k}{10\sqrt{\log(1/\epsilon)}}
\end{equation}
and
\begin{equation}\label{eqp2}
8^r/k^{1/4} \leq 1/100\;.
\end{equation}
Taking $r=2\log\log(k)$, we thus make sure that both (\ref{eqp1}) and (\ref{eqp2}) hold\footnote{Recall that $k \geq 1/\epsilon$. Since Theorem \ref{thm:main} allows us to assume that $\epsilon$ is sufficiently small, we can assume that $k$ is large enough so that $2\log\log k < \frac{\sqrt{\log(k)}}{10}$ and that $8^{2\log\log k}/k^{1/4} \leq 1/100$.} with a lot of room to spare.
Hence, after these $r=2\log\log k$ applications of Lemma \ref{lem:maintrap} we get that ${\cal B}$ must be an $8^{2\log\log k}/k^{1/4}$-refinement of ${\cal P}_{2\log\log k}$.
Since
$$
8^{2\log\log k}/k^{1/4} \leq 1/k^{1/5} \leq \epsilon^{1/5}\;,
$$
we get that ${\cal B}$ is indeed an $\epsilon^{1/5}$-refinement of ${\cal P}_{2\log\log k}$.
\end{proof}

Let us now continue with the proof of Lemma \ref{lem:maintrap}. So throughout the rest of this section we assume all the conditions
that are stated in the lemma. Suppose ${\cal P}_{r-1}=\{X_i: 1 \leq i \leq m \}$ and ${\cal P}_{r}=\{X_{i,i'}: 1 \leq i \leq m, 1 \leq i' \leq M\}$.
Recall the sets $A_{i,j}, B_{i,j}$ that were used in the construction of the graph $G$ in Subsection \ref{subsec:gowconstruct}.
With respect to these, we make the following definition:
\begin{definition} A pair of sets $(Z_t, Z_u)$ is said to be {\em $\beta$-helpful} if
\begin{enumerate}
\item There are\footnote{Note that since $\beta < 1/2$ there is (at most) one choice of $X_i$ and $X_j$ such that $Z_t \subset_{\beta} X_i$ and $Z_u \subset_{\beta} X_j$.}
 $1 \leq i , j \leq m$ such that $Z_t \subset_{\beta} X_i$ and $Z_u \subset_{\beta} X_j$ (we are {\em not} requiring $i \neq j$).
\item We have $\min(|Z_t \cap A_{i,j}|, |Z_t \cap B_{i,j}|) \geq \beta^2 |Z_t|$.
\end{enumerate}
\end{definition}

We will need the following lemma, restated from \cite{Gowers}.

\begin{lemma}{\bf (\cite{Gowers})}\label{lem:balvector}
Let $M$  be an integer and let $(A_j, B_j)_{j=1}^m$ be a sequence of balanced partitions of $[M]$.
Let $0 < \zeta \leq 1/2$ and let $\eta, \xi >0$ be such that
\begin{equation}\label{Glemma}
(1-\eta)(1-4\xi) > 1 - \zeta + \zeta^2\;.
\end{equation}
Then for every sequence $ \lambda = (\lambda_1, \ldots, \lambda_M)$ such that $ \lambda_{i'} \geq 0$ for every $i'$,
$\lVert \lambda \rVert_1 = 1$ and $\lVert \lambda \rVert_{\infty}  < 1 - \zeta$, there are at least $\eta m$ values of $j$
for which $\min (\sum_{i' \in A_j} \lambda_{i'}, \sum_{i' \in B_j} \lambda_{i'}) > \xi$.
\end{lemma}

\begin{lemma}\label{clm:oneset} Suppose ${\cal Z}$ is a $\beta$-refinement of ${\cal P}_{r-1}$.
Then, if $Z_t \subset_{\beta} X_i$ for some $i$, but there is no $i'$ for which $Z_t \subset_{8 \beta}X_{i,i'}$,
then there are at least
$2 \beta m$ sets $X_j$ such that
each of these sets $X_j$ $\beta$-contains at least $\frac{k}{2m}$ sets $Z_u$
such that $(Z_t, Z_u)$ are $\beta$-helpful.
\end{lemma}

\begin{proof} Let $Z_t \subset_{\beta} X_i$ and suppose that there is no $1 \leq i' \leq M$ for which $Z_t \subset_{8 \beta} X_{i,i'}$.
Write $\lambda_{i'}$ for $|Z_t \cap X_{i,i'}|/|Z_t|$. Then $\lambda_{i'} \geq 0$ for all $i'$, $\lVert \lambda \rVert_1 \geq 1 - \beta$
(since $Z_t \subset_{\beta} X_i$) and $\lVert \lambda \rVert_{\infty} \leq 1 - 8 \beta$ (since we assume that there is no $i'$ for which $Z_t \subset_{8\beta}
X_{i,i'}$). Set $\zeta = 7\beta/(1 - \beta) < 1/2$ and note that
we have
\begin{equation}\label{Glemma1}
(1-6\beta)(1-8\beta^2) > 1-6\beta-8\beta^2 > 1-\zeta+\zeta^2\;,
\end{equation}
where in the second inequality we use the fact that $\beta < 1/100$.
Define the vector $\lambda' = \lambda/\lVert \lambda \rVert_1$. Then $\lVert \lambda' \rVert_1 = 1$ and
\begin{equation}\label{Glemma2}
\lVert \lambda' \rVert_{\infty} \leq (1 - 8 \beta)/\lVert \lambda \rVert_{1} \leq  (1 - 8 \beta)/(1- \beta)=1 - \zeta\;.
\end{equation}

Since $(A'_{i,j}, B'_{i,j})^m_{j=1}$ are balanced partitions of $[M]$, we can apply Lemma \ref{lem:balvector}
to the vector $\lambda'$ (with $\eta = 6 \beta$ and $\xi=2\beta^2$), and conclude that
there are at least $6 \beta m$ values of $j$, for which $\min(\sum_{i' \in A'_{i,j}} \lambda'_{i'}, \sum_{i' \in B'_{i,j}} \lambda'_{i'}) > 2 \beta^2$.
Recalling that $\lambda' = \lambda/\lVert \lambda \rVert_1$ and that $\lVert \lambda \rVert_{1} \geq 1 - \beta$ this means that for each such $j$ we have
$\min(\sum_{i' \in A'_{i,j}} \lambda_{i'}, \sum_{i' \in B'_{i,j}} \lambda_{i'}) > 2 \beta^2 (1- \beta) >  \beta^2$.
Notice that by the construction of the sets $A_{i,j}, B_{i,j}$,
(that is $A_{i,j} = \cup_{i' \in A'_{i,j}} X_{i,i'}$ and $B_{i,j} = \cup_{i' \in B'_{i,j}} X_{i,i'}$)
and by the definition of $\lambda$, these $j$'s satisfy
\begin{equation}\label{Glemma3}
\min(|Z_t \cap A_{i,j}|, |Z_t \cap B_{i,j}|) \geq \beta^2 |Z_t|\;,
\end{equation}
that is, they satisfy the second condition of being $\beta$-helpful. This means that if a set $Z_u$ is $\beta$-contained in $X_j$ then
$(Z_t,Z_u)$ is $\beta$-helpful. So to finish the proof, we need to show that out of the $6 \beta m$ values of $j$ that satisfy (\ref{Glemma3}), at least
$2\beta m$ are such that $X_j$ $\beta$-contains at least $k/2m$ sets $Z_u$.
Hence, it is enough to show that ${\cal P}_{r-1}$ has at most $4\beta m$ sets $X$ that $\beta$-contain less than $k/2m$ sets $Z \in {\cal Z}$.

Call a vertex $v \in V(H)$ {\em bad} if it either belongs to a set $Z \in {\cal Z}$ that is
not $\beta$-contained in any $X \in {\cal P}_{r-1}$ or if it belongs
to $Z \setminus X$ where $Z \subset_{\beta} X$.
Note that since we assume that ${\cal Z}$ is a $\beta$-refinement of ${\cal P}_{r-1}$ then the fraction of $H$'s vertices that are
bad is bounded by $2\beta$. Suppose now that there are more than $4\beta m $ sets $X$ that $\beta$-contain less than $k/2m$ sets $Z$.
Recall that each set $X$ contains a $1/m$-fraction of vertices of $H$, while each $Z$ contains a $1/k$-fraction.
Therefore, if $X$ has less than $k/2m$ sets $Z$ that are $\beta$-contained in it, then half of its vertices belong to sets $Z$ that
are either $\beta$-contained in another set $X'$ or that are not $\beta$-contained in any set. Hence, if ${\cal P}_{r-1}$ has more than $4\beta m $ such sets
$X$, then more than $2\beta$-fraction of $H$'s vertices would be bad which is impossible.
\end{proof}

The main part of the proof of Lemma \ref{lem:maintrap} will be the proof of the following lemma

\begin{lemma}\label{clm:key}
Suppose $Z \in {\cal Z}$ and $X_i,X_j \in {\cal P}_{r-1}$. Suppose
$Z \subset_{\beta} X_i$ and there are $\frac{k}{2m}$ sets $Z_u \subset_{\beta} X_j$
such that $(Z, Z_u)$ is $\beta$-helpful. Then at least $\frac{k}{4m}$ of the sets
$Z_u$ are such that $(Z,Z_u)$ is not $\gamma$-regular.
\end{lemma}

We first derive Lemma \ref{lem:maintrap} from Lemmas \ref{clm:oneset} and \ref{clm:key}.

\begin{proof}[Proof of Lemma \ref{lem:maintrap}] By Lemma \ref{clm:oneset} we know that if
$Z_t \subset_{\beta} X_i$ for some $i$, but there is no $i'$ for which $Z_t \subset_{8 \beta}X_{i,i'}$,
then there is $S_t \subseteq [m]$ of size at least $2\beta m$ such that for any $j \in S_t$, the set $X_j$ $\beta$-contains at least
$k/2m$ sets $Z_u$ for which $(Z_t,Z_u)$ is $\beta$-helpful.
By Lemma \ref{clm:key}, each of these sets $X_j$ $\beta$-contains at least $k/4m$ sets $Z_u$ such that $(Z_t,Z_u)$ is not $\gamma$-regular.
Hence, all together (that is, when considering all the sets $X_j$ where $j \in S_t$) there are at least
$\beta k/2$ sets $Z_u$ such that $(Z_t,Z_u)$ is not $\gamma$-regular. Hence, since $\beta^2 > \gamma$ and we assume that ${\cal Z}$ is $\gamma$-regular, there cannot be
more than $2\beta k$ sets $Z_t$ as above.

Since we assume that for at least $(1- \beta)k$ of the sets $Z_t$ there is a set $X_i$ such that
$Z_t \subset_{\beta}X_i$, it follows that for at least $(1 - 3 \beta) k > (1-8\beta)k$ of the sets $Z_t$ there exists an $X_i$ and $i'$ such that $Z_t \subset_{8 \beta}
X_{i,i'}$, which means that ${\cal Z}$ is an $8\beta$-refinement of ${\cal P}_{r}$.
\end{proof}

In the next subsections we complete the proof of Lemma \ref{lem:maintrap} by proving Lemma \ref{clm:key}.

\subsection{Setting the stage for the proof of Lemma \ref{clm:key}}

We start by setting some notation and observing some relations between the parameters involved.
We remind the reader again that we will be assuming the conditions of Lemma \ref{lem:maintrap}.
Also, hereafter we focus only on the $k/2m$ sets $Z_u$
$\subset_{\beta} X_j$ such that $(Z, Z_u)$ are $\beta$-helpful, namely the sets in the statement of
Lemma \ref{clm:key}.

Let us set $A = Z \cap A_{i,j}$ and $B=Z \cap B_{i,j}$. Also for each of the sets $Z_u \subset_{\beta} X_{j}$, if $|Z_u \cap A_{j,i}| \geq |Z_u \cap B_{j,i}|$ we set $W_u = Z_u \cap A_{j,i}$, otherwise
we set $W_u=Z_u \cap B_{j,i}$. Since we assume that all the pairs $(Z,Z_u)$ are $\beta$-helpful and that $\beta \geq \gamma^{1/4}$ we can deduce that
\begin{equation}\label{eqk1}
\min(|A|,|B|) \geq \beta^2|Z| \geq \gamma^{1/2}|Z|\;,
\end{equation}
and for all $u$ we have
\begin{equation}\label{eqk2}
|W_u| \geq (1-\beta)|Z_u|/2 \geq |Z_u|/4\;.
\end{equation}

Let ${\cal P}_{r_1},...,{\cal P}_{r_f}$ be the canonical partitions which refine ${\cal P}_{r-1}$ and on which we have placed a trap.
For each $1 \leq \ell \leq f $, let $\alpha_{\ell}$ be the weight\footnote{So recalling the way we have defined
$H$ in Subsection \ref{subsec:trap}, we get that if $r_{\ell}=b=w(g)$ then $\alpha_{\ell}=4^{-g}$.} that was added to $H$ when placing a trap on partition ${\cal P}_{r_{\ell}}$.
Recall that $H$ contains $\frac{1}{48}\sqrt{\log(1/\epsilon)}$ many traps so
\begin{equation}\label{eqk3}
f \leq \frac{1}{48}\sqrt{\log(1/\epsilon)}\;.
\end{equation}
Also recall that by Fact \ref{trapdensity} we have that all weights $\alpha_1,\ldots,\alpha_{f}$ satisfy
\begin{equation}\label{eqk33}
\alpha_1,\ldots,\alpha_{f} \geq 4^{-\frac{1}{48}\sqrt{\log(1/\epsilon)}}\;.
\end{equation}
Set
\begin{equation}\label{eqk4}
\delta = \frac{4^{-r}}{4^{\sqrt{\log(1/\epsilon)}}}\;,
\end{equation}
and recall that $\delta$ is the extra weight we have added to some of the pairs $(x,y)$ in $G$ when considering partition ${\cal P}_{r-1}$.
Since in Theorem \ref{thm:main} we can assume that $\epsilon$ is sufficiently small, we get from (\ref{eqk3}), (\ref{eqk33}) and (\ref{eqk4}) that
\begin{equation}\label{eqk44}
\delta \ll \frac{1}{f},\alpha_1,\ldots,\alpha_{f}\;.
\end{equation}
We also observe that since $\gamma \leq \epsilon$, and Lemma \ref{lem:maintrap} assumes
that $r \leq \frac{\log(1/\gamma)}{10\sqrt{\log(1/\epsilon)}}$ we get from (\ref{eqk4}) that
\begin{equation}\label{eqk444}
\gamma^{1/3} \ll \delta \;.
\end{equation}

We now define a set $A' \subseteq A$ using the following iterative process. We first set $A_0=A$. If each of the clusters $X \in {\cal P}_{r_1}$ is such that $|A_0 \cap X | < \delta^6|A_0|$, then the process
ends with $A'=A_0$. If there is a cluster $X \in {\cal P}_{r_1}$ such that $|A_0 \cap X | \geq \delta^6|A_0|$ then we set $A_1=|A_0 \cap X|$, and continue to the next phase.
If each of the clusters $X \in {\cal P}_{r_2}$ is such that $|A_1 \cap X | < \delta^6|A_1|$, then the process ends with $A'=A_1$.
If there is a cluster $X \in {\cal P}_{r_2}$ such that $|A_1 \cap X | \leq \delta^6|A_1|$ then we set $A_2=|A_2 \cap X |$ and move to the next phase.
So the process either stops at some level ${\cal P}_{r_t}$ in which none of the clusters of ${\cal P}_{r_t}$ contains more than a $\delta^6$-fraction of $A_{t-1}$, or it goes all the way to ${\cal P}_{r_f}$.

Let us make two important observations about $A'$. First, if the process stops at level ${\cal P}_{r_t}$ (where $t \leq f$) then for any $t' > t$ we have $|A' \cap X | < \delta^6|A'|$ for all
$X \in {\cal P}_{r_{t'}}$. This follows from the fact that ${\cal P}_{r_{t'}}$ refines ${\cal P}_{r_t}$. Therefore, $A'$ has the property, that for each partition ${\cal P}_{r_t}$
the set $A'$ is either contained in a single cluster $X \in {\cal P}_{r_t}$ or none of the clusters contains more than a $\delta^6$-fraction of $A'$.

The second observation is that at each iteration the process picks a subset $A_i$ satisfying $|A_i| \geq \delta^6|A_{i-1}|$.
Since we have at most $f$ iterations, we get that the final set $A'$ we end up with satisfies

\begin{equation}\label{eqk5}
|A'| \geq \delta^{6f}|A| = \left(\frac{4^{-r}}{4^{\sqrt{\log(1/\epsilon)}}}\right)^{6f} |A| \geq_{(\ref{eqk3})} \left(\frac{4^{-r}}{4^{\sqrt{\log(1/\epsilon)}}}\right)^{\frac{6}{48}\sqrt{\log(1/\epsilon)}} |A| \geq \epsilon^{1/4}\gamma^{1/4} |A| \geq \gamma|Z|\;,
\end{equation}
where the third inequality relies on the assumption of Lemma \ref{lem:maintrap} that $r \leq \frac{\log(1/\gamma)}{10\sqrt{\log(1/\epsilon)}}$
and the last uses (\ref{eqk1}) and the fact that $\gamma \leq \epsilon$.
We now use the same process to pick a set $B' \subseteq B$ satisfying the same properties discussed above, and whose size also satisfies
\begin{equation}\label{eqk55}
|B'| \geq \gamma|Z|\;.
\end{equation}

Take one of the sets $W=W_u$ and assume without loss of generality that $W \subseteq A_{j,i}$.
Recall that by $d_G(A',W)$ and $d_G(B',W)$ we denote the densities between these sets in the graph $G$, that is, before adding the traps
to obtain the final graph $H$. First note that since $A',B'$ both belong to $X_i \in {\cal P}_{r-1}$ and $W \subseteq X_j$, we can infer that
exactly the same weight was added in $G$ to $d(A',W)$ and $d(B',W)$ by the partitions ${\cal P}$ that are refined by ${\cal P}_{r-1}$. Now recall that
we put weight $\delta$ between all the edges connecting a vertex in $A_{i,j}$ and a vertex in $A_{j,i}$ and that we did not do so for edges connecting
a vertex in $B_{i,j}$ and a vertex in $A_{j,i}$. Since $A' \subseteq A_{i,j}$, $B' \subseteq B_{i,j}$ and $W \subseteq A_{j,i}$ this means that ${\cal P}_{r-1}$ creates
a discrepancy of $\delta$ between $d_G(A',W)$ and $d_G(B',W)$. Now recall that the weights assigned by $G$ to the partitions ${\cal P}$ which refine ${\cal P}_{r-1}$ are
$\delta/4,\delta/4^2,\delta/4^3,\ldots$. Since the sum of these weights is at most $\delta/3$ we get that
\begin{equation}\label{eqk6}
|d_G(A',W)-d_G(B',W)| \geq \frac23\delta \geq_{(\ref{eqk444})} \gamma \;.
\end{equation}
It thus follows from (\ref{eqk2}) (\ref{eqk5}), (\ref{eqk55}) and (\ref{eqk6}) that if we had not added the traps to $G$, we would have thus concluded that {\em every} $\beta$-helpful pair $(Z,Z_u)$ is not $\gamma$-regular.
So to finish the proof we need to show that a large number of these $\beta$-helpful pairs are not $\gamma$-regular in $H$ as well.

Recall that ${\cal P}_{r_1},...,{\cal P}_{r_f}$ are the partitions which refine ${\cal P}_{r-1}$ and on which we have placed a trap.
For $1 \leq \ell \leq f$ we let $d_{\ell}(A,B)$ be the weight added to $d(A,B)$ by the trap placed on ${\cal P}_{r_{\ell}}$.
We thus have the following claim:

\begin{claim}\label{clm:disc} If $(Z,Z_u)$ is $\gamma$-regular, then there is $1 \leq \ell \leq f$ for
which
\begin{equation}\label{eqk66}
|d_{\ell}(A',W_u)- d_{\ell}(B',W_u)| > 4\delta^2\;.
\end{equation}
\end{claim}

\begin{proof} Recall that since both $A',B' \subseteq X_i \in {\cal P}_{r-1}$ and $W_u \subseteq X_j \in {\cal P}_{r-1}$
we get that $d_H(A',W_u)$ and $d_H(B',W_u)$ get the same weight from each of the traps placed on partitions ${\cal P}_{r'}$
that are refined by ${\cal P}_{r-1}$ (that includes the case that a trap was placed on ${\cal P}_{r-1}$). This means that a discrepancy
between $d_H(A',W_u)$ and $d_H(B',W_u)$ can come either from $d_G(A',W_u)$ and $d_G(B',W_u)$ or from traps placed on partitions
${\cal P}_{r_1},\ldots,{\cal P}_{r_f}$. Thus, if (\ref{eqk66}) does not hold for all $1 \leq \ell \leq f$ then we would have
\begin{eqnarray*}
\left|d_{H}(A',W_u)- d_{H}(B',W_u)\right|&=&\left|d_{G}(A',W_u)- d_{G}(B',W_u)+\sum^{f}_{\ell=1}(d_{\ell}(A',W_u)- d_{\ell}(B',W_u))\right|\\
&\geq& \left|d_{G}(A',W_u)- d_{G}(B',W_u)\right|-\sum^{f}_{\ell=1}\left|(d_{\ell}(A',W_u)- d_{\ell}(B',W_u))\right|\\
&\geq& \frac23\delta-4f\delta^2 \geq_{(\ref{eqk44})} \frac13\delta \geq_{(\ref{eqk444})} \gamma\;,
\end{eqnarray*}
where in the second inequality we use (\ref{eqk6}). Recalling (\ref{eqk2}), (\ref{eqk5}) and (\ref{eqk55}) we thus infer that $(Z,Z_u)$ is not $\gamma$-regular
which is a contradiction.
\end{proof}

Assume that for each $u$ for which $(Z,Z_u)$ is $\gamma$-regular, we set $\ell_u$ to be the {\em smallest} integer for which (\ref{eqk66}) holds.
Recall that $\alpha_{\ell_u}$ is the weight added by the trap placed on the partition ${\cal P}_{r_{\ell_u}}$.
In the following subsection we prove Lemma \ref{clm:key} via Claim \ref{clm:k4} (stated below) and in the subsection following it we prove this claim thus
completing the proof of Lemma \ref{clm:key}.

\begin{claim}\label{clm:k4} If $(Z,Z_u)$ is $\gamma$-regular, then either $A'$ or $B'$ satisfies the following two conditions (we write the condition with respect to $A'$):
\begin{itemize}
\item There is no $X \in {\cal P}_{r_{\ell_u}}$ such that $A' \subseteq X$.
\item $|d_{\ell_u}(A',W_u)-\frac12\alpha_{\ell_u}| > 2\delta^2$.
\end{itemize}
\end{claim}

\subsection{Proof of Lemma \ref{clm:key} via Claim \ref{clm:k4}}\label{subsec:d}

Once again, let us recall that given $Z \subset_{\beta} X_i$ and $X_j$ we are focusing only on
the $k/2m$ sets $Z_u \subset_{\beta} X_j$ such that $(Z, Z_u)$ are $\beta$-helpful.
We need to show that at least $k/4m$ of the sets $Z_u$ are such that $(Z, Z_u)$ is not $\gamma$-regular.

Suppose to the contrary that there are $k/4m$ sets $Z_u$ for which $(Z,Z_u)$ is $\gamma$-regular. Then
by Claim \ref{clm:k4}, for such $Z_u$ either $A'$ or $B'$ satisfies the two conditions of Claim \ref{clm:k4}. Suppose without loss of generality
that in at least $k/8m$ of these cases the set is $A'$. Also, suppose without loss of generality that out of these $k/8m$ cases, in at least $k/16m$
we have $d_{\ell_u}(A',W_u) > \alpha_{\ell_u}/2+2\delta^2$. Finally, since there are only $f$ traps in the canonical partitions
that refine ${\cal P}_{r - 1}$,
we get that there must be an integer $1 \leq \ell \leq f$ for which
there are at least $k/16mf$ sets $W_u$ for which the above holds such that $\ell_u=\ell$. So for each of these sets we have
\begin{equation}\label{eqk8}
d_{\ell}(A',W_u) > \frac12\alpha_{\ell}+2\delta^2\;.
\end{equation}
For what follows we set $S$ to be the collection of $k/16mf$ values of $u$ for which (\ref{eqk8}) holds
and such that $\ell_u=\ell$.

We now make a simple observation which relates $d_{\ell}(A',W_u)$, the graph ${\cal O}_{r_{\ell}}$
that was used to define the trap which was placed on level ${\cal P}_{r_{\ell}}$ and the way in which $A'$ and $W$ are ``spread'' over the clusters of ${\cal P}_{r_{\ell}}$.
Let $m_{r_{\ell}}$ denote the number of clusters of ${\cal P}_{r_{\ell}}$ (which is also the number of vertices of ${\cal O}_{r_{\ell}}$).
Let us use
$Y_1, \ldots, Y_{m_{r_{\ell}}}$
to denote the clusters of ${\cal P}_{r_{\ell}}$.
Suppose $X_i$ and $X_j$ each contain $h$ clusters of ${\cal P}_{r_{\ell}}$.


Let $x^a \in [0,1]^{m_{r_{\ell}}}$ be the vector satisfying $x^a_p=|A' \cap Y_{p}|/|Y_{p}|$ for every $1 \leq p \leq m_{r_{\ell}}$.
Similarly, let $x^u \in [0,1]^{m_{r_{\ell}}}$ be the vector satisfying $x^u_p=|W_u \cap Y_{p}|/|Y_{p}|$ for every $1 \leq p \leq m_{r_{\ell}}$.
If we take $Q$ to be the adjacency matrix of ${\cal O}_{r_{\ell}}$ then

\begin{equation}\label{eqk9}
d_{\ell}(A',W_u)=\frac{(x^a)^T(\alpha_{\ell}Q)x^u}{(\sum_{p}x^a_p)(\sum_{p}x^u_p)}\;.
\end{equation}

Our plan now is to show that the information we have gathered thus far contradicts Lemma \ref{lem:random3}.
Let us start setting the stage for applying this lemma. First, as partition ${\cal P}_{b}$ in Lemma \ref{lem:random3} we will take partition ${\cal P}_{r_{\ell}}$.
So we are using $m_{r_{\ell}}$ as $m_b$ in Lemma \ref{lem:random3}.

Second, as partition ${\cal P}_{b'}$ in Lemma \ref{lem:random3} we will take partition
${\cal P}_{r-1}$. Note that here and in Lemma \ref{lem:random3} we use $m$ to denote the number of clusters in partitions ${\cal P}_{r-1}$ and ${\cal P}_{b'}$ and that
we use $X_1,\ldots,X_m$ to name the $m$ clusters of both partitions. As $\delta$ in Lemma \ref{lem:random3} we use the same $\delta$ used here, that is $\delta = 4^{-r}/4^{\sqrt{\log(1/\epsilon)}}$ as defined in (\ref{eqk4}).
We clearly have $\delta < 1/200$. Also, to satisfy the first condition of Lemma \ref{lem:random3}
we need to make sure that $\delta > 1/\log(m_{r_{\ell}})$, or equivalently that
\begin{equation}\label{eqlogs}
m_{r_{\ell}}=_{(\ref{eq:parorder})}T^{\phi}(r_{\ell}) \geq_{(\ref{eqtower1})} T(\lfloor r_{\ell}/2\rfloor) > 2^{4^{r+\sqrt{\log(1/\epsilon)}}}=_{(\ref{eqk4})}2^{1/\delta}\;,
\end{equation}
We need to verify the second inequality.
Recall that $r_{\ell} \geq r$ since we are only considering traps that were placed on partitions refining ${\cal P}_{r-1}$.
Recalling (\ref{eqw0}) we also have $r_{\ell} \geq \log\log(1/\epsilon)$ since the first trap was placed on the partition with this index.
It is easy to see that these two facts imply that the second inequality in (\ref{eqlogs}) indeed holds.

As the vector $x$ in Lemma \ref{lem:random3} we will take the vector $x^a$ defined above,
and as the vector $y$ we take $\sum_{u \in S}x^u$ with $S$ the set defined just after equation (\ref{eqk8}).
Note that since $A' \subseteq X_i$ and for all $u$ we have $W_u \subseteq X_j$, these vectors satisfy the second condition of Lemma \ref{lem:random3}.

Now, by Claim \ref{clm:k4} there is no cluster\footnote{Recall that we assume that $\ell_u=\ell$ for the set $W_u$ with $u \in S$. See the discussion at the beginning of this subsection.} $X \in {\cal P}_{r_\ell}$ such that $A' \subseteq X$. By the process we have
used to define $A'$, this means that each of the clusters of $X \in {\cal P}_{r_{\ell}}$ contains no more than a $\delta^6$-fraction of the vertices of $A'$.
This means that the vector $x^a$ defined above satisfies the third item of Lemma \ref{lem:random3}.

Finally, observe that each of the sets $Y_{p}$ contains a $1/mh$-fraction of $H$'s vertices\footnote{Since each $X_i$ contains a $1/m$ fraction
of $H$'s vertices and we assumed that $X_i$ is partitioned into $h$ sets $Y_{p}$.} while each set $Z_u$ takes a $1/k$-fraction. We thus get from (\ref{eqk2}) that the sum of entries
of each of the vectors $x^{u}$ is at least $mh/4k$. Since we assume that there are at least $k/16mf$ sets $W_u$, we infer that the
sum of entries of $y$ is at least $h/64f \geq_{(\ref{eqk44})} 2\delta h$. Hence $y$ satisfies the fourth condition of Lemma \ref{lem:random3}.

Since we assume that each of the sets $W_u$ satisfies (\ref{eqk8}), we can use the formulation of (\ref{eqk9}) to infer that
\begin{equation}\label{eqf}
(x^a)^TQx^u > (1/2 + 2\delta^2) \left(\sum_{p}x^a_p\right)\left(\sum_{p}x^u_p\right) = \left(1/2 + 2\delta^2\right) g_1g_2^u\;,
\end{equation}
where we set $g_1=\sum_{p}x^a_p$ and $g^u_2=\sum_{p}x^u_p$. Now set $g_2=\sum_{p}y_p=\sum_ug^u_2$.
Summing over all vectors $x^u$, and applying (\ref{eqf}) we have
$$
(x^a)^TQ y=(x^a)^TQ\left(\sum_u x^u\right) > \left(1/2 + 2\delta^2\right) g_1 \sum_u g^u_2= \left(1/2 + 2\delta^2\right) g_1 g_2\;,
$$
which contradicts (\ref{eqtcond}) in Lemma \ref{lem:random3}.

\subsection{Proof of Claim \ref{clm:k4}}

We recall that we use $\alpha_{\ell}$ to denote the weight
added to $H$ when placing a trap on partition ${\cal P}_{r_{\ell}}$, and that for a set $W_u$ we defined $\ell_u$ just before Claim \ref{clm:k4}.

\begin{claim}\label{clm:k1} Set $\alpha=\alpha_{{\ell_u}}$. If
$
|d_{{\ell}_u}(A',W_u)- d_{{\ell}_u}(B',W_u)| \geq 0.4\alpha
$
then $(Z,Z_u)$ is not $\gamma$-regular.
\end{claim}

\begin{proof} Recall that ${\ell}_u$ was chosen to be the smallest integer for which (\ref{eqk66}) holds. Hence
$$
\left|\sum^{\ell_u-1}_{\ell=1}d_{\ell}(A',W_u) - d_{\ell}(B',W_u)\right| \leq 4f\delta^{2} \leq_{(\ref{eqk44})}  \frac{1}{100}\alpha\;.
$$
The assumption of this claim thus gives
$$
\left|\sum^{\ell_u}_{\ell=1}d_{\ell}(A',W_u) - d_{\ell}(B',W_u)\right| \geq 0.39\alpha\;.
$$
Since the weights assigned to traps with weight smaller than $\alpha$ are given by $\alpha/4,\alpha/16,\ldots$, after taking into account all
the traps placed on ${\cal P}_{r_1},\ldots,{\cal P}_{r_{f}}$ we still have
\begin{equation}\label{eq777}
\left|\sum^{f}_{\ell=1}d_{\ell}(A',W_u) - d_{\ell}(B',W_u)\right| \geq 0.05\alpha\;.
\end{equation}
As we have noted in the proof of Claim \ref{clm:disc}, the only traps that can create a discrepancy between $d_{H}(A',W_u)$ and $d_{H}(B',W_u)$ are those
placed on ${\cal P}_{r_1},\ldots,{\cal P}_{r_f}$. Hence we can disregard the traps that were placed on partitions refined by ${\cal P}_{r-1}$, that is partitions other than ${\cal P}_{r_1},\ldots,{\cal P}_{r_f}$.
Thus, (\ref{eq777}) holds even when considering {\em all} the traps placed in $H$.
Finally, by Fact \ref{ob1} the total weight assigned to edges in $G$ is at most $1/4^{\sqrt{\log(1/\epsilon)}} \leq_{(\ref{eqk33})} 0.01\alpha$.
We thus conclude that
$$
|d_H(A',W_u) - d_H(B',W_u)| \geq 0.04 \alpha >_{(\ref{eqk33})} \epsilon \geq  \gamma\;.
$$
Recalling (\ref{eqk2}), (\ref{eqk5}) and (\ref{eqk55}) we can deduce that $(Z,Z_u)$ is not $\gamma$-regular.
\end{proof}

\begin{claim}\label{clm:k2} If there is a cluster $X \in {\cal P}_{r_{\ell}}$ such that $A' \subseteq X$ and
\begin{equation}\label{eqk7}
\delta^2  \leq d_{\ell}(A',W_u) \leq \alpha_{{\ell}}-\delta^2\;,
\end{equation}
then $(Z,Z_u)$ is not $\gamma$-regular\footnote{Note that in this claim we are not assuming that $\ell=\ell_u$. That is, the claim is true for all $1 \leq \ell \leq f$. However, we will
only apply it with $\ell=\ell_u$.}.
\end{claim}

\begin{proof} Let us define the vectors $x^a$ and $x^u$ as we have done just before equation (\ref{eqk9}).
Let us also use the terminology used when defining these vectors.
So $X=Y_{q}$ for some $Y_{q} \subseteq X_i$ implying
that $x^a_{q}=|A'|/|Y_{q}|$ and all the other entries of $x^a$ are $0$.
Suppose $Y_{1},\ldots,Y_{h}$ are the clusters
of ${\cal P}_{r_{\ell}}$ within $X_j$. Let ${\cal O}_{r_{\ell}}$ be the graph used when placing
the trap on ${\cal P}_{r_{\ell}}$, let $v_{q} \in V({\cal O})$ be the vertex corresponding to cluster $Y_{q}$ and let
$u_1,\ldots,u_h$ be the vertices corresponding to $Y_{1},\ldots,Y_{h}$.
Finally set $N=\{p: (v_{q},u_{p}) \in E({\cal O})\}$ to be the indices of the
vertices $u_1,\ldots,u_h$ which are neighbors of $v_{q}$ in ${\cal O}$.
Then by (\ref{eqk9}) and (\ref{eqk7}) we have
$$
\delta^2  \leq \frac{\alpha_{\ell}\sum_{p \in N}x^u_{p}}{\sum^h_{p=1}x^u_{p}} \leq \alpha_{{\ell}}-\delta^2\;,
$$
implying that
$$
\delta^2  \leq \frac{\sum_{p \in N}x^u_{p}}{\sum^h_{p=1}x^u_{p}} \leq 1-\delta^2\;.
$$
This means that if we take $W^1=W_u \cap (\bigcup_{p \in N}Y_{p})$ then
\begin{equation}\label{eqk10}
\delta^2|W_u| \leq |W^1| \leq (1-\delta^2)|W_u|\;.
\end{equation}
Let $W^2=W_u \setminus W^1$ and note that it satisfies (\ref{eqk10}) as well.
A critical observation now is that our choice of $N$ implies that
for all $p \in N$ the pair $(Y_{q},Y_{p})$ belongs to the trap placed on ${\cal P}_{r_{\ell}}$ and for
all $p \not \in N$ the pair $(Y_{q},Y_{p})$ does not belong to this trap. This means that
$d_{\ell}(A',W^1)=\alpha_{\ell}$ while $d_{\ell}(A',W^2)=0$.

We will now
show that we can find $W' \subseteq W^1$ and $W'' \subseteq W^2$, satisfying $|W^{'}| \geq \epsilon^{1/10}|W^1|$, $|W^{''}| \geq \epsilon^{1/10}|W^2|$ and
\begin{equation}\label{eqlast}
|d_H(A',W')-d_H(A',W'')|\geq \gamma\;.
\end{equation}
Recalling (\ref{eqk5}), this will imply that $(Z,Z_u)$ is not $\gamma$-regular as the fact that $|W'| \geq \epsilon^{1/10}|W^1|$ means that
$$
|W'| \geq \epsilon^{1/10}|W^1| \geq_{(\ref{eqk10})} \epsilon^{1/10}\delta^2|W_u| \geq_{(\ref{eqk2})} \frac14\epsilon^{1/10}\delta^2|Z_u| \geq \frac14\gamma^{1/10}\delta^2|Z_u|  \geq_{(\ref{eqk444})} \gamma|Z_u|\;,
$$
where in the fourth inequality we use the fact that $\gamma \leq \epsilon$.
A similar derivation would show that $|W''| \geq \gamma|Z_u|$.

So we are left with picking the sets $W'$ and $W''$. Let us focus on $W'$.
Consider some $1 \leq \ell' < \ell$. Since we assume that $A'$ is contained is one of the clusters
of ${\cal P}_{r_{\ell}}$ there must be a cluster $Y'_q \in {\cal P}_{r_{\ell'}}$ such that $A' \subseteq Y'_q$.
Take some $p \in N$ and let $Y'_{p} \in {\cal P}_{r_{\ell'}}$ be the cluster containing $Y_{p}$.
So we see that for each pair $(Y_{q},Y_{p})$, either all the vertices $(x,y) \in Y_{q} \times Y_{p}$ get an extra weight of
$\alpha_{\ell'}$ from that trap or none of them do (depending on whether $(Y'_q,Y'_{p})$ belongs to the trap placed on ${\cal P}_{r_{\ell'}}$).
So for each pair $(Y_{q},Y_{p})$ there is a subset $S_{p} \subseteq [\ell-1]$ representing those traps from which
$(Y_{q},Y_{p})$ got an extra weight. Recall now that $H$ contains only $\frac{1}{48}\sqrt{\log(1/\epsilon)}$ many traps, so there
are (much) less than $1/\epsilon^{1/10}$ ways to pick a set $S_p \subseteq [\ell -1]$.
So there must be a subset $N' \subseteq N$ such that $S_{p}=S_{p'}$ for all $p,p' \in N'$ and such that
$|W^1 \cap \bigcup_{p \in N'}Y_{p}| \geq \epsilon^{1/10}|W^1|$. We now take $W'=W^1 \cap \bigcup_{p \in N'}Y_{p}$ and take
$S'$ to be the subset of $[\ell -1]$ which is common to all $p \in N'$. Recapping the above, we see that if $\ell' \in S'$ then $d_{\ell'}(A',W')=\alpha_{{\ell'}}$ and
if $\ell' \not \in S'$ then $d_{\ell'}(A',W')=0$. We can now define $W''$ and $S''$ in a similar way, such that
if $\ell' \in S''$ then $d_{\ell'}(A',W'')=\alpha_{{\ell'}}$ and if $\ell' \not \in S''$ then $d_{\ell'}(A',W'')=0$.

If $S'=S''$ set
$\ell' = \ell$,
otherwise, let $\ell'$ be the smallest index that appears in exactly one of the sets $S'$ and $S''$.
Also, set $\alpha=\alpha_{\ell'}$.
Let us now compare $d_H(A',W')$ and $d_H(A',W'')$. By our choice of $\alpha$, the traps with weight larger than $\alpha$ have the same contribution to both $d_H(A',W')$ and $d_H(A',W'')$. Using again the way we chose $\alpha$ we get that
$$
\left|\sum^{\ell'}_{\ell=1}d_{\ell}(A',W') - d_{\ell}(A',W'')\right|=\alpha\;.
$$
Now observe that the total weight added by traps with weight smaller than $\alpha$ is bounded by $\alpha/4+\alpha/16... < \alpha/3$.
So after taking into account all
traps ${\cal P}_{r_1},\ldots,{\cal P}_{r_f}$ there is still a discrepancy of at least
$$
\left|\sum^{f}_{\ell=1}d_{\ell}(A',W') - d_{\ell}(A',W'')\right| \geq \alpha/2\;.
$$
As in previous proofs, we do not need to consider the weight coming from traps not placed on ${\cal P}_{r_1},\ldots,{\cal P}_{r_f}$ (that is, traps
placed on partitions refined by ${\cal P}_{r-1}$) since $A' \subseteq X_i \in {\cal P}_{r-1}$ and $W_u \subseteq X_j \in {\cal P}_{r-1}$.
Finally, by Fact \ref{ob1}
the total weight assigned to edges in $G$ is bounded by $1/4^{\sqrt{\log(1/\epsilon)}} \leq_{(\ref{eqk33})}\alpha/4$, so after taking into account all the weights
assigned to $(A',W')$ and $(A',W'')$ in $H$ we still have
$$
|d_H(A',W') - d_H(A',W'')| \geq \alpha/4 \geq_{(\ref{eqk33})} \epsilon \geq  \gamma\;.
$$
This proves (\ref{eqlast}) thus completing the proof.
\end{proof}

\begin{claim}\label{clm:k3} If there is a cluster $X \in {\cal P}_{r_{\ell_u}}$ such that $A' \subseteq X$ and
a cluster $Y \in {\cal P}_{r_{\ell_u}}$ such that $B' \subseteq Y$ then $(Z,Z_u)$ is not $\gamma$-regular.
\end{claim}

\begin{proof} If either $A'$ or $B'$ satisfies (\ref{eqk7}) then Claim \ref{clm:k2} implies that $(Z,Z_u)$ is not $\gamma$-regular. So suppose
both do not satisfy (\ref{eqk7}). Now note $d_{\ell_u}(A',W_u), d_{\ell_u}(B',W_u) \leq \alpha_{{\ell_u}}$ since $\alpha_{{\ell}}$ is the maximum
weight a pair of sets can get from the trap placed on ${\cal P}_{r_{\ell}}$. Recall that $\ell_u$ is an integer for which (\ref{eqk66}) holds hence one of the sets (say $A'$) satisfies $0 \leq d_{\ell_u}(A',W_u) \leq \delta^2$ while the other satisfies $\alpha_{{\ell_u}}- \delta^2 \leq d_{\ell_u}(B',W_u) \leq  \alpha_{{\ell_u}}$. But this means that
$$
|d_{{\ell}_u}(A',W_u)- d_{{\ell}_u}(B',W_u)| \geq \alpha_{{\ell_u}}- 2\delta^2 \geq_{(\ref{eqk44})} \alpha_{{\ell_u}}/2\;,
$$
so $(Z,Z_u)$ is not $\gamma$-regular by Claim \ref{clm:k1}.
\end{proof}

We are now ready to complete the proof of Claim \ref{clm:k4}. We know from Claim \ref{clm:k3} that one of the sets $A'$ or $B'$ must satisfy the first requirement of the claim. Suppose it is $A'$.
If $A'$ also satisfies the second item then we are done, so suppose it does not.

If $B'$ also satisfies the first requirement of the claim,
then since $\ell_u$ is chosen to satisfy (\ref{eqk66}) and since we assume that $A'$ does not satisfy the second requirement of the lemma, we get that $B'$ must satisfy the second requirement and we are done.

So suppose now that the $B'$ does not satisfy the first item. If $\delta^2 \leq d_{\ell_u}(B',W_u) \leq \alpha_{\ell_u}-\delta^2$ then by Claim \ref{clm:k2}
$(Z,Z_u)$ is not $\gamma$-regular, which contradicts the assumption of Claim \ref{clm:k4} that $(Z,Z_u)$ is $\gamma$-regular. Finally, if either $d_{\ell_u}(B',W_u) \geq \alpha_{\ell_u}-\delta^2$ or $d_{\ell_u}(B',W_u) \leq \delta^2$ we can combine this with the assumption that $A'$ does not satisfy the second requirement of the claim to get that
$$
|d_{\ell_u}(A',W_u)-d_{\ell_u}(B',W_u)| \geq \frac12\alpha_{\ell_u}-3\delta^2 >_{(\ref{eqk44})} 0.4\alpha_{\ell_u}\;.
$$
Claim \ref{clm:k1} then implies that $(Z,Z_u)$ is not $\gamma$-regular which again contradicts the assumption of Claim \ref{clm:k4}.

\bigskip

\vspace{0.1cm} \noindent {\bf Note added.}\, After completing this paper, we learned that D. Conlon and J. Fox have independently (and simultaneously)
obtained a result similar to the one stated in Theorem \ref{thm:main}. Their proof gives a lower bound of $W(1/\epsilon^{c})$ for some $c>0$ to the strong regularity lemma.


\end{document}